\documentclass[11pt,twoside]{article} %Sep2.2018, zhou
\usepackage{bbm}
\usepackage{amsmath,amssymb}
\usepackage{slashed}
\makeatletter
\newcommand*{\mint}[1]{%
  \mint@l{#1}{}%
}
\newcommand*{\mint@l}[2]{%
  \@ifnextchar\limits{%
    \mint@l{#1}%
  }{%
    \@ifnextchar\nolimits{%
      \mint@l{#1}%
    }{%
      \@ifnextchar\displaylimits{%
        \mint@l{#1}%
      }{%
        \mint@s{#2}{#1}%
      }%
    }%
  }%
}
\newcommand*{\mint@s}[2]{%
  \@ifnextchar_{%
    \mint@sub{#1}{#2}%
  }{%
    \@ifnextchar^{%
      \mint@sup{#1}{#2}%
    }{%
      \mint@{#1}{#2}{}{}%
    }%
  }%
}
\def\mint@sub#1#2_#3{%
  \@ifnextchar^{%
    \mint@sub@sup{#1}{#2}{#3}%
  }{%
    \mint@{#1}{#2}{#3}{}%
  }%
}
\def\mint@sup#1#2^#3{%
  \@ifnextchar_{%
    \mint@sup@sub{#1}{#2}{#3}%
  }{%
    \mint@{#1}{#2}{}{#3}%
  }%
}
\def\mint@sub@sup#1#2#3^#4{%
  \mint@{#1}{#2}{#3}{#4}%
}
\def\mint@sup@sub#1#2#3_#4{%
  \mint@{#1}{#2}{#4}{#3}%
}
\newcommand*{\mint@}[4]{%
  \mathop{}%
  \mkern-\thinmuskip
  \mathchoice{%
    \mint@@{#1}{#2}{#3}{#4}%
        \displaystyle\textstyle\scriptstyle
  }{%
    \mint@@{#1}{#2}{#3}{#4}%
        \textstyle\scriptstyle\scriptstyle
  }{%
    \mint@@{#1}{#2}{#3}{#4}%
        \scriptstyle\scriptscriptstyle\scriptscriptstyle
  }{%
    \mint@@{#1}{#2}{#3}{#4}%
        \scriptscriptstyle\scriptscriptstyle\scriptscriptstyle
  }%
  \mkern-\thinmuskip
  \int#1%
  \ifx\\#3\\\else_{#3}\fi
  \ifx\\#4\\\else^{#4}\fi
}
\newcommand*{\mint@@}[7]{%
  \begingroup
    \sbox0{$#5\int\m@th$}%
    \sbox2{$#5\int_{}\m@th$}%
    \dimen2=\wd0 %
    \let\mint@limits=#1\relax
    \ifx\mint@limits\relax
      \sbox4{$#5\int_{\kern1sp}^{\kern1sp}\m@th$}%
      \ifdim\wd4>\wd2 %
        \let\mint@limits=\nolimits
      \else
        \let\mint@limits=\limits
      \fi
    \fi
    \ifx\mint@limits\displaylimits
      \ifx#5\displaystyle
        \let\mint@limits=\limits
      \fi
    \fi
    \ifx\mint@limits\limits
      \sbox0{$#7#3\m@th$}%
      \sbox2{$#7#4\m@th$}%
      \ifdim\wd0>\dimen2 %
        \dimen2=\wd0 %
      \fi
      \ifdim\wd2>\dimen2 %
        \dimen2=\wd2 %
      \fi
    \fi
    \rlap{%
      $#5%
        \vcenter{%
          \hbox to\dimen2{%
            \hss
            $#6{#2}\m@th$%
            \hss
          }%
        }%
      $%
    }%
  \endgroup
}
\usepackage{mathrsfs}
\usepackage{amssymb}
\usepackage{amsmath}
\usepackage{amsthm}
\usepackage{amsfonts}
\usepackage{color}
\usepackage{graphicx}
\usepackage[active]{srcltx}
\usepackage{tikz}
\usepackage[cp1252]{inputenc}

\usepackage{mathrsfs}
\usepackage{graphicx}

\usepackage[active]{srcltx}

\allowdisplaybreaks

\usepackage{titletoc}
\titlecontents{section}[0pt]{\addvspace{2pt}\filright}
              {\contentspush{\thecontentslabel\ }}
              {}{\titlerule*[8pt]{.}\contentspage}

\def\rr{{\mathbb R}}
\def\rn{{{\rr}^n}}

\def\fz{\infty}

\def\loc{{\mathop\mathrm{\,loc\,}}}

\def\lz{\lambda}
\def\dz{\delta}

\def\ez{\epsilon}

\def\gz{{\gamma}}

\def\bint{{\ifinner\rlap{\bf\kern.35em--}
\int\else\rlap{\bf\kern.45em--}\int\fi}\ignorespaces}

\def\bbint{{\ifinner\rlap{\bf\kern.35em--}
\hspace{0.078cm}\int\else\rlap{\bf\kern.45em--}\int\fi}\ignorespaces}

\newtheorem{thm}{Theorem}[section]
\newtheorem{lem}[thm]{Lemma}%[section]     %@@!!@@!!
\newtheorem{rem}[thm]{Remark}%[section]    %@@!!@@!!
\newtheorem{cor}[thm]{Corollary}%[section]    %@@!!@@!!
\numberwithin{equation}{section}

\title
{\Large\bf   Second order  regularity
  for elliptic and parabolic equations involving $p$-Laplacian
 via a fundamental  inequality
\footnotetext{\hspace{-0.35cm}
\endgraf
This project  was supported by National Natural Science of Foundation of China (No.11522102, 11871088)
and the NSF under the agreement DMS-1600593.
 \endgraf $^\ast$ Corresponding author.
}}

\author{Hongjie Dong, Peng Fa, Yi Ru-Ya Zhang, and Yuan Zhou$^\ast$}
 \date{\today}
\begin{document}
%\begin{CJK*}{GBK}{song}

\arraycolsep=1pt
\allowdisplaybreaks
 \maketitle

\begin{center}
\begin{minipage}{13.5cm}\small
 \noindent{\bf Abstract.}\quad
Denote by $\Delta$  the Laplacian and by $\Delta_\fz $ the $\fz$-Laplacian.  A fundamental inequality is proved for the algebraic  structure of
 $\Delta v\Delta_\fz v$: for every $v\in C^\fz$,
$$\left|  { |D^2vDv|^2} -  {\Delta v \Delta_\fz v } -\frac12[|D^2v|^2-(\Delta v)^2]|Dv|^2\right| \le \frac{n-2}2 [|D^2v|^2{|Dv|^2}-  |D^2vDv|^2 ].
$$
Based on this, we prove the following results:
   \begin{enumerate}
  \item[ (i)]  For any $p$-harmonic functions $u$, $p\in(1,2)\cup(2,\infty)$, we have
$$|Du|^{\frac{p-\gz}2}Du\in W^{1,2}_\loc\quad \mbox{with $\gz<\min\{p+\frac{n-1}{n},3+\frac{p-1}{n-1}\}$}.$$
   As a by-product, when $p\in(1,2)\cup(2,3+\frac2{n-2})$,
 we
    reprove the known $W^{2,q}_\loc$-regularity  of $p$-harmonic functions  for some $q>2$.
 %   Moreover, we show that
%    $$\frac{(p-1)[n-(n-2)(p-2)]}{(p-1)^2+n-1}|D^2u|^2\le|D^2u|^2-(\Delta u)^2,$$
% when $n=2$, which reproves that the map $x\to Du(x)$ is quasi-regular.

  \item[(ii)] When  $n\ge 2$ and  $p\in(1,2)\cup(2,3+\frac2{n-2})$, the  viscosity solutions to parabolic normalized $p $-Laplace equation have
 the $W_\loc^{2,q}$-regularity in the spatial variable and the $W_\loc^{1,q}$-regularity in the time variable
for   some $q>2$.
 Especially, when $n=2$ an open question in \cite{hl} is completely answered.

  \item[(iii)] When $n\ge  1 $ and $p\in(1,2)\cup(2,3)$,  the weak/viscosity solutions to parabolic   $p $-Laplace equation have the $W_\loc^{2,2}$-regularity in the spatial variable and  the $W_\loc^{1,2}$-regularity in the time variable. The range of $p$ (including $p=2$ from the classical result) here is sharp for the $W_\loc^{2,2}$-regularity.
\end{enumerate}

\end{minipage}
\end{center}

%\tableofcontents
%\contentsline{section}{\numberline{}References}{31}

\section{Introduction}

Denote by $\Delta$ and $\Delta_\fz$  the  Laplacian  and  $\fz$-Laplacian, respectively, in $\rn$ with $n\ge2$, i.e.\
 $$\Delta v={\rm div}(Dv)\quad\mbox{and}\quad \Delta_\fz v=D^2vDv\cdot Dv\quad\forall v\in C^\fz.$$
Recall that, the following identity in the plane
\begin{align}\label{keyiden}
 { |D^2vDv|^2} -  {\Delta v \Delta_\fz v } =\frac12[|D^2v|^2-(\Delta v)^2]|Dv|^2\end{align}
is the key to the  higher order Sobolev regularity of infinity harmonic functions in the plane established in \cite{kzz}.
In this paper, we generalize \eqref{keyiden} to the higher dimension: For every $v\in C^\fz$  \begin{equation}\label{keyineq}\left|  { |D^2vDv|^2} -  {\Delta v \Delta_\fz v } -\frac12[|D^2v|^2-(\Delta v)^2]|Dv|^2\right| \le \frac{n-2}2 [|D^2v|^2{|Dv|^2}-  |D^2vDv|^2 ];
\end{equation}
 see Lemma 2.1.

It turns out that \eqref{keyineq} is a basic tool to study the second order Sobolev regularity of equations involving $p$-Laplacian or  its normalization.
Here, for $1<p<\fz$, the $p$-Laplacian $\Delta_p$ and  its normalization $\Delta^N_p$ are defined as
$$\Delta_pv:={\rm div}(|Dv|^{p-2}Dv) \quad \mbox{and}\quad\Delta^N_pv:=|Dv|^{2-p}\Delta_pv,$$
respectively. Throughout the whole paper,  $\Omega$ is always a domain of $\rn$ and $T$ is a positive real number.

  \begin{thm}\label{thm1}
  Let  $n\ge 2$,  $p\in(1,2)\cup(2,\fz)$  and  $\gz<\gz_{n,p}$, where
$\gz_{n,p}:=\min\{p+\frac{n}{n-1},3+\frac{p-1}{n-1}\}.$
For any  weak/viscosity solution $u$ to
 \begin{equation}\label{plap}
\Delta_pu=0 \quad \mbox{in $\Omega$},
 \end{equation}
we have
 $|Du|^{\frac{p-\gz}2}Du\in W^{1,2}_\loc (\Omega) $ and
\begin{equation*}%\label{w22p-gz}
\int_B |D[|Du|^{\frac{p-\gz}2}Du]|^2\,dx\le C(n,p,\gz) \frac1{r^2}\int_{2B}
  |Du|^{p-\gz+2}  \,dx \quad\forall B=B(z,r)\Subset 2B\Subset \Omega.
  \end{equation*}
\end{thm}

Theorem~\ref{thm1} improves the earlier result by Bojarski and Iwaniec  \cite{bi87}, where the convexity and the monotonicity of the $p$-Laplacian were heavily used  in their proof. See Section~\ref{plaplace} for more explanations. As a byproduct,  we reprove the  following  higher integrability of $D^2u$, which was shown earlier by using the Cordes condition (see \cite{ar,mps,mw88}).
 \begin{cor}\label{thm2}
Let $n\ge2$ and $p\in(1,2)\cup(2,3+\frac{2}{n-2})$.
There exists $\dz_{n,p}\in(0,1)$ such that
for any weak/viscosity solution $u$ to \eqref{plap}, we have
   $ D^2u\in L^{q}_\loc(\Omega)$  for  any $q<2+\dz_{n,p}$ and
 \begin{equation}\label{W2q-plap}
 \left(\bint_B    |D^2u |^{q}  \,dx\right)^{1/q}\le C(n,p,q ) \frac1{r}\left(\bint_{2B}
  |Du|^{2}  \,dx\right)^{1/2}\quad\forall B=B(z,r)\Subset 2B\Subset \Omega.
  \end{equation}
  Moreover, %the distributional divergence % $ {\rm div}(D^2uDu-\Delta u Du)=|D^2u|^2-(\Delta u)^2$  with
\begin{align}\label{zz1}
 {\rm div}(D^2uDu-\Delta u Du)=|D^2u|^2-(\Delta u)^2\ge \frac{(p-1)[n-(n-2)(p-2)]}{(p-1)^2+n-1}|D^2u|^2 \quad\mbox{a.e. in $\Omega$.}
\end{align}
  \end{cor}

Similar results also hold  in the parabolic case; some of them were completely open problems. Write $Q_r(z,s):= (s-r^2,s)\times B(z,r)$.
%% Here we keep the $z$ %zhou

\begin{thm}\label{thm3}
Let $n\ge2$ and $p\in(1,2)\cup(2,3+\frac{2}{n-2}) $.
There exists $\dz_{n,p}\in(0,1)$ such that
for any viscosity solution $u=u(x,t)$ to
\begin{equation}\label{pn-plap}
u_t-\Delta^N _pu =0  \quad {\rm in}\ \Omega_T:=\Omega \times (0,T),
\end{equation}
we have
   $ u_t,  D^2u\in L^{q}_\loc(\Omega)$ for any $q<2+\dz_{n,p}$, and
\begin{equation}\label{W2q-pn-plap}\left(\bint_{Q_r}   [|u_t|^{q}+   |D^2u |^{q}]  \,dx\right)^{1/q}\le C(n,p,q ) \frac1{r }\left(\bint_{Q_{2r}}
  |Du|^{2}  \,dx\right)^{1/2}\quad\forall Q_r=Q_r(z,s){ \subset Q_{2r}}\Subset \Omega_T.
  \end{equation}
  \end{thm}

 \begin{rem}\rm
When $n=1$, any solution to \eqref{plap} must be linear and any solution to \eqref{pn-plap} must satisfies the heat equation. Therefore, Theorem \ref{thm1}, Corollary \ref{thm2}, and Theorem \ref{thm3} still hold in the 1D case.
\end{rem}

\begin{thm}\label{thm4}
 Let $n\ge 1$.  For any weak/viscosity solution $u=u(x,t)$
 to \begin{equation}\label{p-plap}
u_t-\Delta _pu =0  \quad {\rm in}\ \Omega_T,
\end{equation}
 the following results  hold.

 \begin{enumerate}
 \item[(i)] For $p\in(1,2)\cup(2,\fz)$, we have $u_t\in L^2_\loc (\Omega_T)$
 and, for any
  $Q_r=Q_r(z,s){ \subset Q_{2r}}\Subset \Omega_T$,
   \begin{align*}%\label{W12p-plap}
\int_{Q_r}(u_t)^2\,dx\,dt\le
\frac{C}{r^2}\int_{{Q_{2r}}}  |Du|^{p}\,dx\,dt+
\frac{C}{r^2}\int_{{Q_{2r}}} |Du| ^{{2p-2}}\,dx\,dt.
\end{align*}

 \item[(ii)] For $p\in(1,2)\cup(2,3)$, we have
    $D^2u\in L^2_\loc (\Omega_T)$ and, for any
  $Q_r=Q_r(z,s){\subset Q_{2r}}\Subset \Omega_T$,
  \begin{align*}%\label{W22p-plap-2}
\int_{Q_r}|D^2 u|^2\,dx\,dt
 &\le   C(n,p)\frac1{r^2}\int_{ Q_{2r} } |Du|^2 \,dx\,dt
+C(n,p)\frac1{r^2}\int_{  Q_{2r} }|Du|^{{4-p}} \,dx\,dt .
\end{align*}
 The range of $p$ (including $p=2$ from the classical result) here is sharp for the $W_\loc^{2,2}$-regularity.
  \end{enumerate}
\end{thm}

\begin{rem}\rm
 By keeping track of the constants, it is clear that the implicit constants $C$ in the all above results do not blow up as $p\to 2$.
\end{rem}

In the following subsections, we introduce the background and related results for all types of the equations considered  above in details, and give more remarks on our results.

\subsection{$p$-Laplace equation and its normalization}\label{plaplace}
To start with, we consider the $p$-Laplace equation \eqref{plap}.
A function $u:\Omega\to\rr$ is  $p$-harmonic  provided that $u\in W^{1,p}(\Omega)$  is a  weak solution   to  \eqref{plap},
 that is,
 $$
 \int_\Omega |Du|^{p-2}Du\cdot D\phi\,dx=0\quad\forall \phi\in C^\fz_c(\Omega).$$
Recall  that
 $p$-harmonic functions are identified with   viscosity solutions to \eqref{plap}  by Juuntinen-Lindqvist-Manfredi \cite{jlm01} (see also Julin-Juuntinen \cite{jj12}), and also identified  with
viscosity solutions to   $\Delta_p^Nu =0$  in $\Omega$ by Peres-Sheffield \cite{ps}.

Formally, we have
$$\Delta^N_p v=\Delta  v+(p-2)\frac{\Delta_\fz v}{|Dv|^2}\quad \mbox{and} \quad\Delta_p v=|Dv|^{p-2}[
\Delta v+(p-2)\frac{\Delta_\fz v}{|Dv|^2}].$$
Therefore, the normalized $p$-Laplace operator can be regarded as an ``interpolation'' between Laplacian and (normalized) $\infty$-Laplacian. This  was, indeed, rigorously interpreted by the theory of stochastic tug-of-war games;
see  \cite{ps} and also \cite{pssw}.

It was well-known that any $p$-harmonic function  belongs to $C^{1,\alpha}$
 for some $\alpha\in(0,1)$ depending on $n$ and $p$, but not necessarily $C^{1,1}$ when $p>2$; see Uraltseva \cite{u68}, Lewis \cite{l83}, Dibenedetto \cite{d82}, Evans \cite{e82} and also Uhlenbeck \cite{u77}.

Regarding Theorem~\ref{thm1}, let us recall that Bojarski-Iwaniec  \cite{bi87}  proved that $|Du|^{\frac{p-2}2}Du\in W^{1,2}_\loc $ for all $p$-harmonic functions $u$; see also  Uraltseva \cite{u68} when $p\in(2,\fz)$. In their proof,
certain convexity and the monotonicity of  the $p$-Laplace operator is heavily utilized. Hence, by   $|Du|\in L^\fz_\loc$, $|Du|^{\frac{p-\gz}2}Du\in W^{1,2}_\loc $ for all $\gz\le 2$.
In this paper,  however, without using any convexity or the monotonicity of  the $p$-Laplace operator but only with \eqref{keyineq}, we improve the range  $\gz\le 2$  to  $\gz<\gz_{n,p}$ in Theorem \ref{thm1}. In particular, the range is improved to $ \gz<p+2$ when $n=2$, which we conjecture to be optimal.
Note that when $n=2$,
$$
\gz_{n,p}=p+2= p+\frac{n}{n-1}= 3+\frac{p-1}{n-1},
$$
and when $n\ge3$, $$\gz_{n,p}=3+\frac{p-1}{n-1}<p+\frac{n}{n-1} \ \mbox{ if $p>2$},\quad\mbox{and}\quad
\gz_{n,p}=p+\frac{n}{n-1}<3+\frac{p-1}{n-1}\ \mbox{ if $p<2$}.$$

The $W^{2,q}_\loc$-regularity in Corollary \ref{thm2}  was  known via the so-called {\it Cordes condition}. The Cordes condition was introduced to study the summability of the second derivative
for second order linear operators in non-divergence form with measurable coefficients; see Cordes \cite{c61}, Talenti \cite{t65}, Campanato \cite{c67}  and also Maugeri-Palagachev-Softova \cite{mps}. We also refer the reader to Bers-Nirenberg \cite{bn54}, Caffarelli-Cabr\'e \cite{cc95}
and Lin \cite{l86} for general study.
Manfredi-Weitzman \cite{mw88} used the Cordes condition to study $p$-harmonic functions so
to get the $W^{2,2}_\loc$-regularity when $1<p<3+\frac2{n-2}$, and then one can get the $W^{2,q}_\loc$-regularity for some $q>2$  via the argument therein and \cite[Theorem 1.2.1\&1.2.3]{mps}. For a brief explanation of the  application of the Cordes condition, see  Remark \ref{plap-cor} (i) (see also \cite[Theorem 4.1]{ar}).
We also note that it is not enough to get Theorem \ref{thm1} and also \eqref{zz1} in Corollary \ref{thm2} via the Cordes condition; see Remark \ref{plap-cor} (ii).

 We remark  that when $n=2$ and $p\in(1,2)\cup(2,\fz)$, for any $p$-harmonic function $u$ in $\Omega$, \eqref{zz1} gives
$$|D^2u|^2\le -\frac{(p-1)^2+1}{p-1}\det D^2u\quad\mbox{a.e. in $\Omega$.}$$
This implies that the map $x\to Du(x)$ is  quasi-regular, which was originally obtained by Bojarski-Iwaniec \cite{bi87}. When $n\ge 3$ and $p\in(1,2)\cup (2,3+\frac2{n-2})$,  the nonnegativity of $|D^2u|^2-(\Delta u)^2$ given in  \eqref{zz1} is new.
When $n\ge 3$ and $p\in(3+\frac2{n-2},\fz)$,
we conjecture that $|D^2u|^2-(\Delta u)^2$ changes sign for some $p$-harmonic function $u\in W^{2,2}_\loc$. For more discussions see Remark 3.5.

Finally, we remark that, when $n=2$ and $p\in(1,\fz)$,
  via hodograph  method Iwaniec-Manfredi \cite{im89} showed the
   $ C^{k,\alpha}(\Omega)\cap W^{k,q}_\loc$-regularity of $p$-harmonic functions,
where ranges of $k,\alpha$  and $q$ are sharp. But when $n\ge3$ and $p\in[3+\frac2{n-2},\fz)$,
 it remains open to get $u\in W^{2,2}_\loc(\Omega)$, in other words,  to improve the range $\gz\in(-\fz,\gz_{n,p})$ in Theorem \ref{thm1} to $\gz\in(-\fz,p]$.

\subsection{Parabolic normalized $p$-Laplace equation}

The  parabolic  normalized $p$-Laplace equation \eqref{pn-plap} is closely related to the
theory of stochastic tug-of-war games, and has certain applications in   economics and image process, see e.g.\  Manfredi-Parviainen-Rossi\cite{mpr},  Does \cite{d11}, Nystr\"om-Parviainen \cite{np14},  and Elmoataz-Toutain-Tenbrinck \cite{ett}.

For any viscosity solution
 to   \eqref{pn-plap},
Does \cite{d11} and Banerjee-Garofalo \cite{bg,bg15} established their interior Lipschitz
regularity in  the spatial variables.
However, the interior Lipschitz
regularity in   the time variable is   open
unless certain assumptions are added on the  behavior at the lateral boundary.
Jin and Silvestre \cite{js} proved   the
 $C^{1,\alpha}$-regularity    in the spatial variables and the
 $C^{0,(1+\alpha)/2} $-regularity  in the time variable for some $\alpha\in(0,1)$.
 We also refer the reader  to  \cite{ijs,ap} for analogue results for
 general parabolic equations involving $\Delta^N_p$.

H{\o}eg and Lindqvist \cite{hl} established  the  $W^{2,2}_\loc$-regularity in the spatial variables for viscosity solutions  to \eqref{p-plap}
when $\frac65<p<\frac{14}5$ and the
$W^{1,2}_\loc$-regularity in the time variable  when $1<p<\frac{14}5$.
The limits $\frac65$ and $\frac{14}5$ are evidently an artifact of their method,
and their method is not capable to reach all $p\in(1,\fz)$.
They also explained that, through the Cordes condition  (see e.g.  \cite[(1.106)]{mps} for parabolic version),
it is possible to get analogue results for $p$ in some restricted range    but not all $p\in(1,\fz)$, mainly since
the absence of zero (lateral) boundary values produces many  undesired terms  which are hard  to estimates. Indeed,  by the parabolic version of the Cordes condition, the only possible range to get the $W^{2,2}_\loc$-regularity  is   $p\in(1,3+\frac{2}{n-1})$; see  Remark \ref{pn-plap-cor} for details.

Additionally, the  following question was raised by H{\o}eg and Lindqvist   \cite{hl}:

\noindent {\it
For all $p\in(1,\fz)$, whether viscosity solutions  to \eqref{pn-plap} enjoy  the $W^{2,2}_\loc$-regularity in the spatial variables and the $W^{1,2}_\loc$-regularity in the time variable?}

The higher integrability of second order derivative was also completely open.
Theorem 1.3 not only completely answers this questions when $n=2$, but also improves the result by H{\o}eg and Lindqvist \cite{hl} when $n\ge3$  by getting better range of $p$ with higher order integrability.

\subsection{Parabolic $p$-Laplace equation}

Finally, we focus on the parabolic $p$-Laplace equation \eqref{p-plap}.
For  the equivalence of the weak and  viscosity solutions to \eqref{p-plap}  we refer to \cite{jlm01,jj12}. Recall that the  $C^{1,\alpha}$-regularity for weak/viscosity solutions to \eqref{p-plap} was established by DiBenedetto-Friedman \cite{df85} (see also Wiegner \cite{w86}).
  The $L^2_\loc$-integrality of $u_t$ is easy to get from the divergence structure of the equation \eqref{p-plap}.
However,  to the best of our knowledge, so  far no second order regularity in the spatial variables has been  known in general. We note that the approach via the parabolic version of the Cordes condition  does not work here; see Remark \ref{p-plap-cor}.

The  range $p\in(1,3)$ in Theorem \ref{thm4} is optimal
to get the $W^{2,2}_\loc$-regularity. Indeed, the function
$$ w(x_1,x_2)=\frac{p^{p-1}}{(p-1)^{p-1}}t+|x_1|^{1+\frac{1}{p-1}}$$   is  a viscosity solution to \eqref{p-plap} in $\rr^2\times(0,\fz)$, but a direct calculation shows that $$\mbox{$|D^2w|=C|x_1|^{\frac {2-p}{p-1}}\in L^2_\loc(\rr^2\times(0,\fz))$ if and only  if $ p<3$.}$$

 \subsection{Ideas of the proofs}

The proofs of Theorem  \ref{thm1} and Corollary \ref{thm2} are given in Sections 3.
 Let $u$ be a  $p$-harmonic function  in $\Omega$. For any smooth domain $U\Subset\Omega$, we  consider the smooth approximation function $u^\ez$ with $\ez\in(0,1]$, which is the  solution to
\begin{equation}\label{ap-plap}
 {\rm div}\left([|D u^\ez|^2+\ez]^{\frac{p-2}2 } D u^\ez\right)=0\quad {\rm in}\
U ;\ u^\ez=u\quad {\rm on}\ \partial U.
\end{equation}
 Applying \eqref{keyineq} to $u^\ez$, in Section 3, via a direction calculation one has
 \begin{align}\label{pi-plap2}
\frac{n}{2}|D|Du^\ez||^2 +\frac{1}{p-2} \frac{(\Delta u^\ez)^2}{|Du^\ez|^2} [|Du^\ez|^2+\ez]
\le\frac{1}{2}[|D^2u^\ez|^2-(\Delta u^\ez)^2]+\frac{n-2}{2}|D^2u^\ez|^2
\quad\mbox{a.e. in $U$}
\end{align} and
 \begin{align}\label{pi-lpap}
&[\frac{n}{2(p-2)^2}+\frac{1}{p-2}-\frac{n-2}{2}] |D^2u^\ez|^2 \le[\frac{n}{2(p-2)^2}+\frac{1}{p-2}+\frac{1}{2}] [|D^2u^\ez|^2-(\Delta u^\ez)^2]\quad \mbox{in $U$.}
\end{align}

Since $1<p<3+\frac2{n-2}$ implies  $$\frac{n}{2(p-2)^2}+\frac{1}{p-2}-\frac{n-2}{2}>0,$$
from \eqref{pi-lpap} and   the divergence structure of $|D^2u^\ez|^2-(\Delta u^\ez)^2$ (see Lemma \ref{div}) we deduce
   $$ \bint_B    |D^2u^\ez |^2  \,dx \le C(n,p  ) \frac1{r^2}\inf_{{{\vec c}}\in\rn}\bint_{2B}
  |Du^\ez-{{\vec c}}|^{2}  \,dx \quad\forall B\subset 2B\Subset U.$$
  Sending $\ez\to 0$ and using the Sobolev-Poincar\'e inequality, by Gehring's lemma    we  obtain  \eqref{W2q-plap}. As a by-product, one   gets \eqref{zz1}.

To get Theorem \ref{thm1}, multiplying both sides of \eqref{pi-plap2} by $[|D u^\ez|^2+\ez]^{\frac{p-\gz}2}\phi^2$ for any $\phi\in C^\fz(\Omega)$ and integrating, if $\gz<\gz_{n,p}$, after some calculation we obtain
$$\frac{|D^2u^\ez  Du^\ez |^2}{|Du^\ez|^2+\ez} [|D u^\ez|^2+\ez]^{\frac{p-\gz}2}\in L^1_\loc(U)\quad
\mbox{and}\quad (\Delta u^\ez)^2[|D u^\ez|^2+\ez]^{\frac{p-\gz}2} \in L^1_\loc(U)$$
 uniformly in $\ez>0$. Further calculation  yields that $$|D[|D u^\ez|^2+\ez]^{\frac{p-\gz}4}Du^\ez|^2   \in L^1_\loc(U)\ \mbox{uniformly in $\ez>0$}.$$
Sending $\ez\to0$ one concludes $|D[|Du| ^{\frac{p-\gz}2}Du]|\in L^2_\loc$ as desired.

Theorem \ref{thm3} is proved in Section 4. Let $u=u(x,t)$ be a viscosity solution
to  \eqref{pn-plap}. Given any  smooth domain $U\Subset \Omega$, for $\ez\in(0,1]$
we let $u^\ez\in C^\fz(U) \cap C^0(\overline U)$  be a viscosity  solution to   the regularized   equation
\begin{equation}\label{ap-pn-plap}
 u^\ez_t- \Delta u^\ez-(p-2) \frac{\Delta _\fz u^\ez}{  |D u^\ez|^2+\ez }=0\quad {\rm in}\
U ;\ u^\ez=u\quad {\rm on}\ \partial_p U_T.
\end{equation}
Applying \eqref{keyineq} to $u^\ez$, one gets
 \begin{align*}%\label{pi-pn-plapx}
&\frac{n}{2}\frac{|D^2u^\ez Du^\ez|^2}{|Du^\ez|^2+\ez}+[\frac{1}{p-2}
-\frac{n-2}{2}](\Delta u^\ez)^2
\nonumber\\
&\quad\le\frac{n-1}{2}[|D^2u^\ez|^2-(\Delta u^\ez)^2]-\frac{\ez}{2}\frac{|D^2u^\ez|^2-(\Delta u^\ez)^2}{|Du^\ez|^2+\ez}+ \frac{\Delta u^\ez u^\ez_t}{p-2}\quad\mbox{in $U_T$}.
\end{align*}
Compared to \eqref{pi-plap2} or \eqref{pi-lpap} for approximation functions to $p$-harmonic functions, here we have the additional
 term $$\frac{\Delta u^\ez u^\ez_t}{p-2}$$  from the parabolic structure,  and also an
 annoying  term $$-\frac{\ez}{2}\frac{|D^2u^\ez|^2-(\Delta u^\ez)^2}{|Du^\ez|^2+\ez}$$  from the approximation procedure, either of which cannot be removed.
With additional ideas and careful/tedious calculations
  to bound such two terms (see Section 4.1 and Section 4.2
by considering two cases via different  methods), we are able to prove in Lemma \ref{du-cpll} that,
  if    $p\in(1,2)\cup(2, 3+\frac{2}{n-2})$, then
$$
 \int_{Q_r}[|D^2u^\ez|^2+|u^\ez_t|^2]\,dx\,dt
 \le C(n,p)\frac1{r^2}\inf_{{\vec c}\in\rn}\int_{ Q_{2r}}|Du^\ez-{{\vec c}}|^2 \,dx\,dt+   o(1)
\quad\forall Q_r\subset { Q_{2r}}\Subset U_T.
$$
From this, the parabolic Sobolev-Poincar\'e inequality and Gehring's Lemma, we conclude \eqref{W2q-pn-plap}.

Finally, we prove Theorem \ref{thm4} in Section 5. Let $u=u(x,t) $ be a viscosity solution to  \eqref{p-plap}.
Given any smooth domain $U\Subset \Omega$,   for  $\ez\in(0,1]$ we
let $u^\ez\in C^\fz(U)\cap C^0(\overline U)$  be a weak solution to   the regularized   equation
\begin{equation}\label{ap-par-plap}
 u^\ez_t-{\rm div}\left([|D u^\ez|^2+\ez]^{\frac{p-2}2 } D u^\ez\right)=0\quad {\rm in}\
U ;\ u^\ez=u\quad {\rm on}\ \partial_p U_T.
\end{equation}

To obtain Theorem \ref{thm4}, it suffices to show that $D^2u^\ez,\, u^\ez_t\in L^2_\loc(U_T)$ uniformly in $\ez>0$.  Note that, from the divergence structure of \eqref{ap-par-plap}, one  easily gets  $u^\ez_t\in L^2_\loc(U_T)$ uniformly in $\ez>0$.
To see   $D^2u^\ez \in L^2_\loc(U_T)$ uniformly in $\ez>0$,  we apply \eqref{keyineq} to $u^\ez$ so to  get
\begin{align*}%\label{pi-pn-plap-1x}
&[\frac{n}{2(p-2)^2}-\frac n2]|D^2u^\ez|^2   \nonumber  \\
&\quad\le[ \frac{n}{2(p-2)^2}-\frac 12][|D^2u^\ez|^2-(\Delta u^\ez)^2]\nonumber\\
&\quad\quad- \frac{n-2p+4}{(p-2)^2} [\frac12(u^\ez_t)^2(|Du^\ez|^2+\ez)^{2-p}- \Delta u^\ez u^\ez_t(|Du^\ez|^2+\ez)^{\frac{2-p}2}]\nonumber \\
&\quad\quad+\{-(\Delta u^\ez)^2-(p-2)\frac{(\Delta_\fz u^\ez)^2}{[|Du^\ez|^2+\ez]^2}+\frac{\Delta u^\ez u^\ez_t}{p-2}[|Du^\ez|^2+\ez]^{\frac{2-p}2}+ \frac{\ez}{2}\frac{(\Delta u^\ez)^2}{|Du^\ez|^2+\ez}\}.
\end{align*}
Observe that $p\in(1,2)\cup(2,3)$ implies
$$
\frac{n}{2(p-2)^2}-\frac n2>0.
$$
By bounding  the integration of the last two terms (see Lemmas \ref{th} and \ref{th5-1-1}),    we conclude $u^\ez\in W^{2,2}_\loc(U_T)$ uniformly in $\ez>0$.

\section{Structures for $\Delta v\Delta_\fz v$ and $|D^2v|^2-(\Delta v)^2$}

The following algebraic structural inequality for $\Delta v\Delta_\fz v$ plays a central role in this paper.
\begin{lem}\label{keylem} Let $n\ge2$ and $U $ be a domain of $\rn$. For any $v\in C^\fz(U )$, we have
 \begin{align}  &\left|  { |D^2vDv|^2} -  {\Delta v \Delta_\fz v } -\frac12[|D^2v|^2-(\Delta v)^2]|Dv|^2\right|\nonumber\\
 &\quad\quad \quad\quad \quad\quad\quad\quad\quad\quad\quad\le \frac{n-2}2 [|D^2v|^2{|Dv|^2}-  |D^2vDv|^2 ]\quad \mbox{in $U $}.\label{keyineqx}
 \end{align}
\end{lem}
To prove this we need the following result.
\begin{lem} \label{keylem0} For any vector  ${\vec \lz}=(\lz_1,\ldots,\lz_n)$ and $\vec a=(a_1,\ldots,a_n)\in\rn$ with $|\vec a|=1$,  we have
 \begin{align*} & \left| \sum_{i=1}^n (\lz_i)^2(a_i)^2 -  (\sum_{i=1}^n \lz_i) [\sum_{j=1}^n \lz_j(a_j)^2] -\frac12[\sum_{i=1}^n (\lz_i)^2-(\sum_{i=1}^n \lz_i)^2] \right| \\
 &\quad\le
 \frac{n-2}2 [\sum_{i=1}^n (\lz_i)^2-  \sum_{i=1}^n (\lz_i)^2(a_i)^2 ].
 \end{align*}
\end{lem}
\begin{proof}
Write
  \begin{align*}
  &\sum_{i=1}^n (\lz_i)^2(a_i)^2 -  (\sum_{i=1}^n \lz_i) [\sum_{j=1}^n \lz_j(a_j)^2]
   = -\sum_{i=1}^n \sum_{j\ne i} \lz_i   \lz_j(a_j)^2
     = -\sum_{j=1}^n (a_j)^2 \sum_{i\ne j} \lz_i   \lz_j
    \end{align*}
 Given any $j=1,\ldots, n$,  write
$$ \sum_{i\ne j} \lz_i   \lz_j =
\sum_{1\le i<j} \lz_i   \lz_j +\sum_{ j<k\le n}    \lz_j\lz_k= \sum_{i=1}^{n-1} \sum_{ i< k\le n}\lz_i\lz_k- \sum_{i=1}^{j-1} \sum_{ i<k\ne j} \lz_i\lz_k -\sum_{i=j+1}^{n-1} \sum_{ i< k\le n} \lz_i\lz_k .$$
Since
$$\sum_{i=1}^{n-1} \sum_{ i< k\le n}\lz_i\lz_k =\frac12 [\sum_{i=1}^n (\lz_i)^2-  (\sum_{i=1}^n  \lz_i)^2 ],$$
by using $\sum_{j=1}^n  (a_j)^2=1$,
we have
 \begin{align*} -\sum_{j=1}^n (a_j)^2 \sum_{i\ne j} \lz_i   \lz_j=
 \frac12 [\sum_{i=1}^n (\lz_i)^2-  (\sum_{i=1}^n  \lz_i)^2 ]-   \sum_{j=1}^n (a_j)^2 (\sum_{i=1}^{j-1} \sum_{ i<k\ne j}  +\sum_{i=j+1}^{n-1} \sum_{ i< k\le n}) \lz_i\lz_k .
    \end{align*}
By the Cauchy-Schwarz inequality,
  \begin{align*} \left|(\sum_{i=1}^{j-1} \sum_{ i<k\ne j}  +\sum_{i=j+1}^{n-1} \sum_{ i< k\le n}) \lz_i\lz_k \right| &\le\frac12(\sum_{i=1}^{j-1} \sum_{ i<k\ne j}% [(\lz_i)^2+(\lz_k)^2]  + \frac 12
 + \sum_{i=j+1}^{n-1} \sum_{ i< k\le n}) [(\lz_i)^2+(\lz_k)^2] \\
&= \frac{n-2}2\sum_{i\ne j}(\lz_i)^2\\
&= \frac{n-2}2\sum_{i=1}^n(\lz_i)^2 -\frac{n-2}2(\lz_j)^2.
\end{align*}
Using $\sum_{j=1}^n  (a_j)^2=1$ again, we conclude
$$\left |-   \sum_{j=1}^n (a_j)^2 (\sum_{i=1}^{j-1} \sum_{ i<k\ne j}  +\sum_{i=j+1}^{n-1} \sum_{ i< k\le n}) \lz_i\lz_k\right|\le \frac{n-2}2[\sum_{i=1}^n(\lz_i)^2 - \sum_{j=1}^n(\lz_j)^2(a_j)^2].$$
Combining the inequalities above, we get \eqref{keyineqx} as desired.
\end{proof}

We are ready to prove Lemma \ref{keylem}.

\begin{proof}[Proof of Lemma \ref{keylem}] Let  $\bar x\in U $.
If $Dv(\bar x)=0$,   then obviously   \eqref{keyineqx} holds.
We   assume that $Dv(\bar x)\ne 0$ below.  By  dividing both sides by $|Dv|^2$, we further assume that $|Dv(\bar x)|=1$.

At  $\bar x$, $D^2v$ is a symmetric matrix  and hence
 its eigenvalues are given by  $\{\lz_i\}_{i=1}^n\subset\rr$.
 One may find an orthogonal matrix $O\in {\bf O}(n)$ so that
$$O^TD^2vO={\rm diag}\{\lz_1,\lz_2,\ldots,\lz_n\}.$$
Note that $O^{-1}=O^T$.
At $\bar x$, it follows that
$$
|D^2v|^2=|O^TD^2v O|^2=\sum_{i=1}^n(\lz_i)^2,\quad \Delta v=\sum_{i=1}^n \lz_i.
$$
Writing $O^TDv=\sum_{i=1} a_i {\bf e}_i=:a$,   we have
$$\Delta_\fz v=(Dv)^TD^2vDv=(O^T Dv)^T (O^T D^2v O) (O^TDv)=\sum_{i=1}^n \lz_i(a_i)^2  $$
and $$|D^2vDv|^2= |(O^T D^2v O)(O^TDv)|^2=\sum_{i=1}^n (\lz_i)^2(a_i)^2.$$
Applying Lemma \ref{keylem0} to $\vec\lz $ and
$\vec a $, we obtain
 \begin{align*} &\left|  { |D^2vDv|^2} -  {\Delta v \Delta_\fz v } -\frac12[|D^2v|^2-(\Delta v)^2]|Dv|^2\right|\\
  &\quad =\left| \sum_{i=1}^n (\lz_i)^2(a_i)^2 -  (\sum_{i=1}^n \lz_i) [\sum_{j=1}^n \lz_j(a_j)^2] -\frac12[\sum_{i=1}^n (\lz_i)^2-(\sum_{i=1}^n \lz_i)^2] \right| \\
 &\quad\le
 \frac{n-2}2 [\sum_{i=1}^n (\lz_i)^2-  \sum_{i=1}^n (\lz_i)^2(a_i)^2 ]\\
 &\quad= \frac{n-2}2 [|D^2v|^2{|Dv|^2}-  |D^2vDv|^2 ]
 \end{align*}
 as desired.
 \end{proof}

We also need the following divergence structure of $|D^2v |^2-(\Delta v )^2$.
Below we use Einstein's summation convention, that is, $a_ib_i=\sum_{i=1}^na_ib_i$.

\begin{lem}\label{div} For  any  $v\in C^\fz(U)$, $\phi\in C_c^\fz(U )$
and vector ${{\vec c}}\in\rn$, we have
   \begin{align}\label{du-c}
  \left|\int_U [|D^2v |^2-(\Delta v )^2] \phi^2\,dx\right|
 \le C  \int_{U }  |Dv-{{\vec c}}|^2 [|\phi||D^2\phi|+|D\phi|^2]\,dx.
 \end{align}
\end{lem}

\begin{proof}
First, we note that
\begin{align}\label{e8.03}|D^2v |^2-(\Delta v  )^2={\rm div}(D^2v   Dv -\Delta v   Dv ) \quad\mbox{in $U$.}\end{align}
Via integration by parts, a direct calculation leads to
\begin{align*}
 \int_U[|D^2v  |^2-(\Delta v  )^2] \phi^2\,dx
 &=       -  2 \int_{U} (v  _{{x_ix_j}}v  _{x_i}-v  _{x_ix_i}v  _{x_j}) \phi_{x_j}\phi \,dx.
\end{align*}
For any vector ${{\vec c}}=(c_1,c_2,...,c_n)\in\rn$, since
  $$ -  2 \int_{U} (v  _{{x_ix_j}}c_{ i}-v  _{x_ix_i}c_{ j}) \phi_{x_j}\phi \,dx=2 \int_{U} (v  _{{x_j}}c_{ i}-v  _{x_i }c_{j }) (\phi_{x_j}\phi)_{x_i}\,dx=0,$$
  one has
  \begin{align*}
 &\int_U[|D^2v |^2-(\Delta v )^2] \phi^2\,dx\\
 &\quad=       -  2 \int_{U} [v _{{x_ix_j}}(v _{x_i}-c_i)-v _{x_ix_i}(v _{x_j}-c_j)] \phi_{x_j}\phi \,dx\\
 &\quad =-   2\int_{U}  (v _ {x_i}-c_i)_{x_j} (v _{x_i}-c_i)  \phi_{x_j}\phi \,dx
 +2\int_{U} (v _{x_i}-c_i)_{x_i}(v _{x_j}-c_j)  \phi_{x_j}\phi \,dx.
 \end{align*}
Using integration by parts,
 \begin{align*} -   2\int_{U}  (v _ {x_i}-c_i)_{x_j} (v _{x_i}-c_i)  \phi_{x_j}\phi \,dx& =-   \int_{U} (|Du-c|^2)_{x_j} \phi_{x_j}\phi \,dx =
  \int_{U}  |Du-c|^2 (\phi_{x_j}\phi)_{x_j}\,dx
 \end{align*} and \begin{align*}
  &2\int_{U} (v _{x_i}-c_i)_{x_i}(v _{x_j}-c_j)  \phi_{x_j}\phi \,dx\\
  &\quad=
-
  2\int_{U} (v _{x_i}-c_i)(v _{x_j}-c_j)_{x_i}  \phi_{x_j}\phi \,dx-
  2\int_{U} (v _{x_i}-c_i) (v _{x_j}-c_j)  (\phi_{x_j}\phi)_{x_i} \,dx
  \\
  &\quad= -
  2\int_{U} (v _{x_i}-c_i)(v _{x_i}-c_i)_{x_j}  \phi_{x_j}\phi \,dx-
  2\int_{U} (v _{x_i}-c_i) (v _{x_j}-c_j)  (\phi_{x_j}\phi)_{x_i} \,dx.
 \end{align*}
 Combining these and using the Cauchy-Schwarz inequality, we conclude \eqref{du-c}.
%We conclude that
%  \begin{align*}
% &\left|\int_U[|D^2v |^2-(\Delta v )^2] \phi^2\,dx\right| \\
% & \quad=  \left|\int_{U}  |Du-c|^2 (\phi_{x_j}\phi)_{x_j}\,dx-
% 2\int_{U} (v _{x_i}-c_i) (v _{x_j}-c_j)  (\phi_{x_j}\phi)_{x_i} \,dx\right| \\
% &\quad\le C \int_{U}  |Du-c|^2 [|\phi||D^2\phi|+|D\phi|^2]\,dx
% \end{align*}
%as desired.
\end{proof}

\section{Proofs of Theorem \ref{thm1} and Corollary \ref{thm2}}

Let $u$ be a $p$-harmonic function in $\Omega$.
Given any   smooth domain $U\Subset \Omega$,   for $\ez\in(0,1]$
we let $u^\ez\in W^{1,p}(U)\cap C^0(\overline U)$  be a weak solution to   the regularized   equation
\eqref{ap-plap}.
By the elliptic theory, we know that $u^\ez\in C^\fz(U )\cap C^0(\overline U )$,
$Du^\ez\in L^\fz(U)$ uniformly in $\ez>0$ and $u^\ez\to u$ in $C^0(U)$  as $\ez\to0$; see \cite{u68,l83,d82}.

Applying \eqref{keyineqx} to $u^\ez$, we claim the following two inequalities:
\begin{align}\label{pi-plapx}
&[\frac{n}{2(p-2)^2}+\frac{1}{p-2}-\frac{n-2}{2}] |D^2u^\ez|^2 \le[\frac{n}{2(p-2)^2}+\frac{1}{p-2}+\frac{1}{2}] [|D^2u^\ez|^2-(\Delta u^\ez)^2]\quad \mbox{in $U$}
\end{align}
and
\begin{align}\label{alden1}
\frac{n}{2}|D|Du^\ez||^2+\frac{1}{p-2} \frac{(\Delta u^\ez)^2}{|Du^\ez|^2} [|Du^\ez|^2+\ez]
\le\frac{1}{2}[|D^2u^\ez|^2-(\Delta u^\ez)^2]+\frac{n-2}{2}|D^2u^\ez|^2
\end{align}
whenever $|Du^\ez|$ is differentiable and hence almost everywhere in $U$.
 Note that here, $u^\ez\in C^\fz(U)$ implies $|Du^\ez|$ is locally Lipschitz in $U$, and hence, by Rademacher's theorem, $D|Du^\ez|$ exists almost everywhere in $U$.
Moreover, at a point $\bar x \in U$, if $Du^\ez(\bar x )=0$, then we may always
 set
\begin{equation}\label{e3.z3} \frac{(\Delta u^\ez)^2}{|Du^\ez|^\alpha}=0\quad\mbox{for any $0\le \alpha<4$.}
\end{equation}
Indeed,
\begin{align*}
(\Delta u^\ez)^2=\frac{(p-2)^2(\Delta_\fz u^\ez)^2}{[|Du^\ez|^{ 2}+\ez]^2}
\le (p-2)^2\frac{|D^2u^\ez|^2|Du^\ez|^4}{[|Du^\ez|^  2 +\ez]^2}=O(|Du^\ez|^4) \ \mbox{whenever $x\to\bar x$.}
\end{align*}

\medskip
{\noindent\it Proofs of \eqref{pi-plapx} and \eqref{alden1}.}
Given any point $\bar x\in U$, if $Du^\ez(\bar x )= 0$, then by \eqref{e3.z3}, one has
\eqref{pi-plapx}.  If $Du^\ez(\bar x )= 0$ and also
$|Du^\ez|$ is also differentiable at $\bar x$, then $D|Du^\ez|(\bar x)=0$.
By \eqref{e3.z3} again, \eqref{alden1} holds at $\bar x$.

Below we assume  $Du^\ez(\bar x )\ne 0$.
Observe that $|Du^\ez|$ is differentiable at $\bar x$ and
\begin{equation}\label{e3.z2}|D|Du^\ez|(\bar x )|=\frac{|D^2u^\ez(\bar x )Du^\ez(\bar x )|}{|Du^\ez|(\bar x )}.
\end{equation}
 On the other hand, applying \eqref{keyineqx} to $u^\ez$ and employing the non-divergence form of \eqref{ap-plap}, at $\bar x $ one gets
 \begin{align*}
&{ |D^2u^\ez Du^\ez |^2} + \frac{(\Delta u^\ez )^2}{p-2} [|Du^\ez|^2+\ez]   -\frac12[|D^2u^\ez|^2-(\Delta u^\ez)^2]|Du^\ez|^2  \\
&\quad \le \frac{n-2}2 [|D^2u^\ez|^2{|Du^\ez |^2}-  |D^2u^\ez Du^\ez|^2 ].
 \end{align*}
Dividing both sides by $|Du^\ez(\bar x )|^2$, at $\bar x $ we get
\begin{equation}\label{e3.z1} \frac n2 \frac{ |D^2u^\ez Du^\ez |^2}{|Du^\ez|^2} +  \frac1{p-2}\frac{(\Delta u^\ez )^2} {|Du^\ez|^2} [|Du^\ez|^2+\ez]   \le \frac12[|D^2u^\ez|^2-(\Delta u^\ez)^2]   +\frac{n-2}2 |D^2u^\ez|.
\end{equation}
 From this  and \eqref{e3.z2}, one concludes \eqref{alden1} at $\bar x$ as desired.

Moreover, at $\bar x $, employing the non-divergence form of \eqref{ap-plap}   and H\"older's inequality, one has
 $$\frac{ |D^2u^\ez Du^\ez |^2}{|Du^\ez|^2}\ge
\frac{ |\Delta_\fz u^\ez |^2}{|Du^\ez|^4}\ge  \frac1{(p-2)^2}\frac{(\Delta u^\ez)^2}{|Du^\ez|^2} [|Du^\ez|^2+\ez].$$
From this and \eqref{e3.z1}, it follows that
\begin{align}
                    \label{eq7.01}
&[\frac{n}{2(p-2)^2}+\frac{1}{p-2}] \left(\frac{\Delta u^\ez}{|Du^\ez|}\right)^2[|Du^\ez|^2+\ez] \le\frac{1}{2} [|D^2u^\ez|^2-(\Delta u^\ez)^2] +\frac{n-2}{2} |D^2u^\ez|^2.
\end{align}
Since
$$\frac{n}{2(p-2)^2}+\frac{1}{p-2}>0,$$    \eqref{eq7.01}  gives
\begin{align*}
&[\frac{n}{2(p-2)^2}+\frac{1}{p-2}] (\Delta u^\ez)^2 \le\frac{1}{2} [|D^2u^\ez|^2-(\Delta u^\ez)^2] +\frac{n-2}{2} |D^2u^\ez|^2.
\end{align*}
 Adding both sides by
 $$[\frac{n}{2(p-2)^2}+\frac{1}{p-2}][|D^2u^\ez|^2-  (\Delta u^\ez)^2],$$
 we obtain \eqref{pi-plapx}.

 By using \eqref{alden1} and Lemma \ref{div}, we prove Corollary \ref{thm2} as  follows.
   \begin{proof}[Proof of Corollary \ref{thm2}]
Since $1<p<3+\frac2{n-2}$, we have
\begin{equation}\label{e3.x8}
0<\frac{n}{2(p-2)^2}+\frac{1}{p-2}-\frac{n-2}{2}<\frac{n}{2(p-2)^2}+\frac{1}{p-2}+\frac{1}{2}.
\end{equation}
From this,  \eqref{pi-plapx}  and Lemma \ref{div}  we conclude  that
for any $\phi\in C^\fz_c(U)$,
\begin{align*}\int_U|D^2u^\ez|^2\phi^2\,dx
\le C(n,p) \inf_{{\vec c}\in\rn}\int_U|Du^\ez-{{\vec c}}|^2[|D\phi|^2+|\phi|
|D^2\phi|]\,dx.
\end{align*}
 By choosing suitable test function $\phi$, we obtain
 \begin{equation*}%\label{W22-plap-1}
 \bint_B    |D^2u^\ez |^2  \,dx \le C(n,p ) \inf_{{{\vec c}}\in\rn}\frac1{r^2}\bint_{2B}
  |Du^\ez-{{\vec c}}|^{2}  \,dx\quad\forall B=B(z,r)\Subset 2B\Subset U.
  \end{equation*}
  This together with $Du^\ez\in L^\fz(U)$ uniformly in $\ez>0$, implies that
$u^\ez\in W^{2,2}_\loc(U)$ uniformly in $\ez>0$. By the compact embedding theorem,
$u^\ez\to u$ in $W^{1,q}_\loc(U)$ for $1<q<2n/(n-2)$  and weakly in $W^{2,2}_\loc(U)$ as $\ez\to0$.
Letting $\ez\to0$, we conclude
\begin{equation*}%\label{W22-plap}
\bint_B    |D^2u |^2  \,dx \le C(n,p ) \inf_{{{\vec c}}\in\rn}\frac1{r^2}\bint_{2B}
  |Du-{{\vec c}}|^{2}  \,dx\quad\forall B=B(z,r)\Subset 2B\Subset U.
  \end{equation*}
Applying the Sobolev-Poincar\'e inequality, one has
    $$\left(\bint_B    |D^2u |^2  \,dx\right)^{1/2}\le C(n,p )  \left( \bint_{2B}
  |D^2u|^{\frac{2n}{n+2}}   \,dx\right)^{\frac{n+2}{2n}}\quad\forall B\Subset 2B\Subset U.$$
 Via Gehring's lemma (see for example \cite{g73,g}), we therefore conclude that there exists a $\dz_{n,p}>0$ such that
 $D^2u\in L^q_\loc(\Omega)$ for any $q<2+\dz_{n,p}$ and
    $$\left(\bint_B    |D^2u |^q  \,dx\right)^{1/q}\le C(n,p,q ) \left(\bint_B    |D^2u |^2  \,dx\right)^{1/2} \quad\forall B\Subset 2B\Subset U.$$
This gives \eqref{W2q-plap}.

To see \eqref{zz1}, let
$$K_{n,p}:=\frac{\frac n{2(p-2)^2}+\frac1{p-2}+\frac12}
{\frac n{2(p-2)^2}+\frac1{p-2}-\frac{n-2}2}=\frac{(p-1)^2+n-1}{(p-1)[n-(n-2)(p-2)]}.$$
From \eqref{pi-plapx}, \eqref{e3.x8},
$Du^\ez\to Du$ in $L^2_\loc(U)$ and weakly in $W^{1,2}_\loc(U)$ and \eqref{e8.03}, one  deduces  that
\begin{align*} \frac1{K_{n,p}}\int_U|D^2u |^2\phi^2\,dx&\le  \frac1{K_{n,p}}\lim_{\ez\to0}\int_U|D^2u^\ez |^2\phi^2\,dx\\
&
\le  \lim_{\ez\to0}\int_U[|D^2u^\ez |^2-(\Delta u^\ez)^2]\phi^2\,dx\\
&=  \lim_{\ez\to0}\int_U(D^2u^\ez Du^\ez-\Delta u^\ez Du^\ez) D(\phi^2)\,dx \\
&=   \int_U(D^2u  Du -\Delta u  Du ) D(\phi^2)\,dx\quad\forall \phi\in C^\fz_c(U),
\end{align*}
that is,
\begin{align}\label{e8.02}
\frac1{K_{n,p}} |D^2u |^2\le {\rm div }(D^2u  Du -\Delta u  Du ).
\end{align}

On the other hand, since $u\in W^{2,q}_\loc$, %by approximating $u$ via $u\ast\psi_\dz$,
letting $\{\psi_\dz\}_{\dz>0}$ be the standard smooth mollifier,  one has
\begin{align*} \int_U(D^2uDu-\Delta uDu)\cdot D\phi\,dx&=
\lim_{\dz\to0}\int_U[D^2(u\ast\psi_\dz)D(u\ast\psi_\dz)-\Delta (u\ast\psi_\dz)D(u\ast\psi_\dz)]\cdot D\phi\,dx\\
&=
\lim_{\dz\to0}\int_U[|D^2(u\ast\psi_\dz)|^2-|\Delta (u\ast\psi_\dz) |^2] \phi\,dx\\
&= \int_U[|D^2u|^2-|\Delta u |^2] \phi\,dx\quad\forall \phi\in C^\fz_c(U),
\end{align*}
which implies that
the distributional divergence
%${\rm div}(D^2uDu-\Delta uDu)$ is given by $|D^2u|^2-(\Delta u)^2$, that is,
\begin{align}\label{e8.01}{\rm div}(D^2uDu-\Delta uDu)=|D^2u|^2-(\Delta u)^2.
\end{align}
Obviously, \eqref{zz1} follows from  \eqref{e8.02} and \eqref{e8.01}.
\end{proof}

To prove Theorem \ref{thm1}, we use \eqref{pi-lpap} and also, instead of Lemma \ref{div}, the following result.
\begin{lem} \label{divp-gz} For any $\gz\in\rr$, $\eta>0$ and $\phi\in C^\fz_c(U)$, we have
\begin{align}\label{e3.x4}
 \int_U[|D^2u^\ez|^2-(\Delta u^\ez)^2][|D u^\ez|^2+\ez]^{ \frac{p-\gz}2 }\phi^2\,dx
 & \le    -   (p-\gz-\eta) \int_{U}  \frac{|D^2u^\ez Du^\ez|^2}{|Du^\ez|^2+\ez} [|D u^\ez|^2+\ez]^{ \frac{p-\gz}2 }\phi^2\,dx\nonumber\\
 & \quad    -[
   \frac{p-\gz}{p-2}-\eta]\int_{U} (\Delta u^\ez)^2  [|D u^\ez|^2+\ez]^{  \frac{p-\gz}2 }\phi^2\,dx\nonumber\\
    &\quad+ C(n,\eta) \int_{U }       [|D u^\ez|^2+\ez]^{\frac{p-\gz+2}2 }|D\phi|^2\,dx.
\end{align}
\end{lem}
\begin{proof}

Via integration by parts, a direct calculation leads to
\begin{align*}
 & \int_U[|D^2u^\ez|^2-(\Delta u^\ez)^2][|D u^\ez|^2+\ez]^{ \frac{p-\gz}2 }\phi^2\,dx\\
  & \quad=-\int_U( u^\ez_{x_ix_j} u^\ez_{x_i} -\Delta u^\ez  u^\ez_{x_j} )[|D u^\ez|^2+\ez]^{ \frac{p-\gz}2 }\phi^2]_{x_j}\,dx\\
 &\quad=    -   (p-\gz) \int_{U}  \frac{|D^2u^\ez Du^\ez|^2}{|Du^\ez|^2+\ez} [|D u^\ez|^2+\ez]^{ \frac{p-\gz}2 }\phi^2\,dx \\
  &\quad\quad  +
   (p-\gz) \int_{U}  \Delta u^\ez  \frac{\Delta_\fz u^\ez}{|Du^\ez|^2+\ez} [|D u^\ez|^2+\ez]^{  \frac{p-\gz}2 }\phi^2\,dx\\
 &\quad\quad  -
  2\int_{U} (u^\ez_{{x_ix_j}}u^\ez_{x_i}-\Delta u^\ez u^\ez_{x_j}) \phi_{x_j}\phi[|D u^\ez|^2+\ez]^{\frac{p-\gz}2}\,dx.
\end{align*}
By Young's inequality, for any $ \eta>0$,
  \begin{align*}
&\int_{U} (u^\ez_{{x_ix_j}}u^\ez_{x_i}-\Delta u^\ez u^\ez_{x_j}) \phi_{x_j}\phi[|D u^\ez|^2+\ez]^{\frac{p-\gz}2}\,dx\\
&\quad\le   \eta  \int_{U }  \frac{|D^2u^\ez Du^\ez|^2} {|Du^\ez|^2+\ez}
  [|D u^\ez|^2+\ez]^{\frac{p-\gz}2 }      \phi^2  \,dx   + \eta  \int_{U }   (\Delta u^\ez)^2
  [|D u^\ez|^2+\ez]^{\frac{p-\gz}2 }      \phi^2  \,dx\\
      &\quad\quad+ C(n)\frac1\eta\int_{U }       [|D u^\ez|^2+\ez]^{\frac{p-\gz+2}2 }|D\phi|^2\,dx.
    \end{align*}
Since $\frac{\Delta_\fz u^\ez}{|Du^\ez|^2+\ez}=\frac{\Delta u^\ez}{p-2}$,
we get \eqref{e3.x4} as desired.
\end{proof}
From   \eqref{alden1} and Lemma \ref{divp-gz} one deduces the following.
\begin{lem}\label{l3.x5}
If  $p\in(1,2)\cup(2,\fz)$ and $\gz<\gz_{n,p}$, then
 \begin{align}\label{e3.y4}&\int_U
 \frac{ |D^2u^\ez Du^\ez|^2}{|Du^\ez|^2+\ez}  [{|Du^\ez|^2+\ez}]^{\frac{p-\gz}2 }\phi^2\,dx+
 \int_U
 (\Delta u^\ez )^2    [{|Du^\ez|^2+\ez}]^{\frac{p-\gz}2  }\phi^2\,dx\nonumber\\
 &\quad\le
C(n,p,\gz) \int_{U }       [|D u^\ez|^2+\ez]^{\frac{p-\gz+2}2}|D\phi|^2\,dx \quad\forall \phi\in C^\fz_c(U). &\end{align}
 \end{lem}

\begin{proof}
From \eqref{alden1} and \eqref{e3.x4}, one has
\begin{align*}  L &:=[\frac n2+\frac{(n-1)(p-\gz)}2-\eta]\int_U
 \frac{ |D^2u^\ez Du^\ez|^2}{|Du^\ez|^2+\ez} [{|Du^\ez|^2+\ez}]^{\frac{p-\gz}2}\phi^2\,dx\\
  &\quad\quad+[\frac1{p-2}-\frac{n-2}2+\frac{(n-1)(p-\gz)}{2(p-2)}-\eta]\int_U {(\Delta u^\ez)^2}[{|Du^\ez|^2+\ez}]^{\frac{p-\gz}2}\phi^2\,dx\\
   &\le   C(n)\frac1\eta\int_{U }       [|D u^\ez|^2+\ez]^{\frac{p-\gz+2}2 }|D\phi|^2\,dx.
  \end{align*}
By the non-divergence form of \eqref{ap-plap},
  $$\frac{(\Delta u^\ez)^2}{(p-2)^2} =
 \frac{|\Delta_\fz u^\ez|^2}{[|D u^\ez|^2+\ez]^2}\le  \frac{|D^2u^\ez Du^\ez|^2|Du^\ez|^2}{[|D u^\ez|^2+\ez]^2}\le \frac{|D^2u^\ez Du^\ez|^2 }{ |Du^\ez|^2 +\ez}.$$
Since $\gz<p+\frac{n}{n-1}$ implies  $$\frac n2+\frac{(n-1)(p-\gz)}2>0,$$
for $0<\eta<\frac14[\frac n2+\frac{(n-1)(p-\gz)}2]$  we have
   \begin{align*}
 L&\ge \eta \int_U
 \frac{ |D^2u^\ez Du^\ez|^2}{|Du^\ez|^2+\ez}  [{|Du^\ez|^2+\ez}]^{\frac{p-\gz}2 }\phi^2\,dx\\
 &\quad+ [c(n,p,\gz)- \frac2{(p-2)^2}\eta-\eta]
  \int_U
  (\Delta u^\ez)^2  [{|Du^\ez|^2+\ez}]^{\frac{p-\gz}2 }\phi^2\,dx,
  \end{align*}
  where
  \begin{align*}
c(n,p,\gz)&:=\frac1{(p-2)^2}[\frac n2+\frac{(n-1)(p-\gz)}2]+\frac1{p-2}-\frac{n-2}2+\frac{(n-1)(p-\gz)}{2(p-2)}\\
&= \frac{p-1}{2(p-2)^2}[ (n-1)(p-\gz)-(n-2)(p-2)+n ].
\end{align*}
Since $\gz<3+\frac{p-1}{n-1}$ implies  $$(n-1)(p-\gz)>(p-3)(n-1)-(p-1)=(n-2)(p-2)-n,$$  we have
 $c(n,p,\gz)>0$.

Choosing $\eta>0$ so that
 $$
 c(n,p,\gz)-\frac2{(p-2)^2}\eta-\eta>\eta,$$
 we  get the desired \eqref{e3.y4}.
 \end{proof}
 As a consequence of Lemma \ref{divp-gz} and Lemma \ref{l3.x5} we obtain
  \begin{cor}  \label{c3.x6}
If  $p\in(1,2)\cup(2,\fz)$ and $\gz<\gz_{n,p}$, then
\begin{align}\label{e3.x7}
\int_U |D([|D u^\ez|^2+\ez]^{\frac{p-\gz}4}Du^\ez )|^2 \phi^2\,dx
      &\le C( n,p,\gz)\int_{U}    [|D u^\ez|^2+\ez]^{\frac{p-\gz+2}2}|D\phi|^2\,dx\quad\forall \phi\in C^\fz_c(U).
   \end{align}
\end{cor}
\begin{proof}
Note that
   \begin{align*}&|D([|D u^\ez|^2+\ez]^{\frac{p-\gz}4} Du^\ez)|^2\\&\quad=
[|D u^\ez|^2+\ez]^{\frac{p-\gz}2}\left| D^2u^\ez +   \frac{p-\gz}2 \frac{Du^\ez \otimes  D^2u^\ez Du^\ez}{|Du^\ez|^2+\ez}\right|^2\\
&\quad=[|D u^\ez|^2+\ez]^{\frac{p-\gz}2}[| D^2u^\ez|^2+ (p-\gz)\frac{|D^2u^\ez Du^\ez|^2}{|Du^\ez|^2+\ez}
+ \frac{(p-\gz)^2}4 \frac{|Du^\ez|^2 | D^2u^\ez Du^\ez|^2}{[|D u^\ez|^2+\ez]^2}]\\
&\quad\le C(n,p,\gz)[|D u^\ez|^2+\ez]^{\frac{p-\gz}2} | D^2u^\ez|^2.
\end{align*}
By Lemma \ref{divp-gz}, for any $\eta>0$,
\begin{align*}
\int_U |D^2u^\ez|^2[|D u^\ez|^2+\ez]^{ \frac{p-\gz}2 }\phi^2\,dx&\le [  - (p-\gz)+\eta ] \int_{U}  \frac{|D^2u^\ez Du^\ez|^2}{|Du^\ez|^2+\ez} [|D u^\ez|^2+\ez]^{ \frac{p-\gz}2 }\phi^2\,dx \\
 &\quad + [1-
   \frac{p-\gz}{p-2}+\eta ] \int_{U} (\Delta u^\ez)^2  [|D u^\ez|^2+\ez]^{  \frac{p-\gz}2 }\phi^2\,dx\\
   &\quad + C( n,p,\eta)\int_{U}    [|D u^\ez|^2+\ez]^{\frac{p-\gz+2}2}|D\phi|^2\,dx\\
   &\le C( n,p,\gz)\int_{U}    [|D u^\ez|^2+\ez]^{\frac{p-\gz+2}2}|D\phi|^2\,dx,
   \end{align*}
where in the last inequality we took $\eta=1$ and used Lemma \ref{l3.x5}.  Combining the above two inequalities, we get \eqref{e3.x7} as desired.
 %  \begin{align*}
%\int_U |D^2u^\ez|^2[|D u^\ez|^2+\ez]^{ \frac{p-\gz}2 }\phi^2\,dx&
%   &\le C( n,p,\gz)\int_{U}    [|D u^\ez|^2+\ez]^{\frac{p-\gz+2}2}|D\phi|^2\,dx.
%   \end{align*}
%This gives Corollary \ref{c3.x6}.
\end{proof}

Now we are able to prove Theorem \ref{thm1}.
 \begin{proof}[Proof of Theorem \ref{thm1}.]
 Since $|Du^\ez|\in L^\fz_\loc(U)$, by Corollary \ref{c3.x6} we have
$[|D u^\ez|^2+\ez]^{\frac{p-\gz}4}Du^\ez\in W^{1,2}_\loc(U)$ uniformly in $\ez>0$.
By the weakly compactness of $W^{1,2}_\loc(U)$, as $\ez\to0$ (up to some subsequence),
   $[|D u^\ez|^2+\ez]^{\frac{p-\gz}4}Du^\ez $ converges to some function $\vec v$   in $L^2_\loc(U)$
   and weakly in $W^{1,2}_\loc(U)$. Since $Du^\ez\to Du$ in $C^{0,\alpha}(U)$ for some $\alpha>0$, we get
  $ \vec v=|Du|^{\frac{p-\gz}2}Du$ and hence
 $$ \int_{B}|D[|Du|^{\frac{p-\gz}2}Du ]|^2\,dx\le C(n,p)\frac1{r^2}\int_{2B}|Du|^{p-\gz+2}\,dx\quad\forall B=B(z,r)\Subset 2B\Subset U$$
as desired.
\end{proof}

Below, we  give some remarks about  the Cordes condition.

\begin{rem} \label{plap-cor}\rm (i)
The $W^{2,q}_\loc$-regularity in  Corollary \ref{thm2} was proved  via the Cordes condition previously. Precisely, rewrite the equation \eqref{ap-plap} as
$$\sum_{1\le i,j\le n}a^\ez_{ij}u^\ez_{x_ix_j}=0\quad\mbox{in $U$};\ u^\ez=u\ \mbox{on $\partial U$},$$
where the  coefficients
$$a^\ez_{ij}=\dz_{ij}+(p-2)\frac{u^\ez_{x_i}u^\ez_{x_j}}{|Du^\ez|^2+\ez}.$$
If $1<p<3+\frac2{n-2}$, then $\{a^\ez_{ij}\}_{1\le i,j\le n}$ satisfies the Cordes condition uniformly in
  $\ez\in(0,1]$, that is, there exists an $\dz>0$ such that
 $$\sum_{i,j=1}^n(a^\ez_{ij})^2\le \frac1{n-1+\dz}\left(\sum_{i=1}^na^\ez_{ii}\right)^2 \quad  \mbox{ in $U$ for all   $\ez\in(0,1]$}.$$
Applying \cite[Theorem 1.2.1]{mps}, Manfredi-Weitzman \cite{mw88} showed that $u^\ez \in W^{2,2}_\loc(U)$ uniformly in $\ez>0$.
  Indeed, following \cite[Theorem 1.2.3]{mps} and the arguments in \cite{mw88} (see also \cite{ar}),
   one could  get the $u^\ez\in W^{2,q}$-regularity uniformly in $\ez>0$  for some $q>2$.
  Letting $\ez\to0$, one has $u\in W^{2,q}_\loc(\Omega)$.

  (ii)   The   Coders condition  is not valid for \eqref{zz1} in Corollary \ref{thm2}, and  also Theorem \ref{thm1} in general.
  \end{rem}

  We end this section by the following remark for \eqref{zz1} in Corollary \ref{thm2}.
  \begin{rem}\rm
  (i) Let  $n=2$. Note that for $v\in W^{2,2}_\loc$,
  one has   $$|D^2v|^2-(\Delta v)^2=- 2 \det D^2v \quad a.e.$$
For $1<p<\fz$, by \eqref{zz1}, and also by the property of harmonic functions when $p=2$,
 one  has
\begin{equation}\label{e8.03b} |D^2u|^2\le - \frac{(p-1)^2+1}{p-1}\det D^2u\quad a.e.
\end{equation}
whenever  $u$  is a planar $p$-harmonic function.
This  implies that the map $x\to Du(x)$ is quasi-regular, which was originally proved by \cite{bi87}.
 The constant in \eqref{e8.03b} is sharp. In fact, consider the boundary value problem $\Delta_p u=0$ in $B_1\subset {\mathbb R}^2$ with a boundary condition $u=\varphi$ on $\partial B_1$ which is even with respect to $x_2$. Then by the uniqueness of solutions, we know that $u$ is even in $x_2$, and thus $D_2 u=D_{12}u=0$ on $x_2=0$. Now by using the equation, it is easily seen that $D_{22} u=(1-p)D_{11} u$ on $x_2=0$, so that the equality in \eqref{e8.03b} holds.
Moreover,  in the limiting case $p=\fz$, it was shown in \cite{kzz} that
 $$ \mbox{$-\det D^2u$ is a nonnegative Radon measure and $|D|Du||^2 \le -\det D^2 u$}$$ whenever $u$ is a planar $\fz$-harmonic function.

(ii) Let $n\ge3$.
For $p\in(1,3+\frac2{n-2})$, by \eqref{zz1} and theory of harmonic functions,  we see that
$|D^2u|^2-(\Delta u)^2 $   is nonnegative.
Observe that
$$  \Delta(|Du|^2)-(u_{x_i}u_{x_j})_{x_ix_j} = {\rm div}(D^2uDu-\Delta uDu)=|D^2u|^2-(\Delta u)^2$$
in the sense of distributions. When $p=3+\frac2{n-2}$,
 we  expect that the  distributional  second order derivative  $ \Delta(|Du|^2)-(u_{x_i}u_{x_j})_{x_ix_j}  $ is a nonnegative Radon measure.
On the other hand,  when  $p=\fz$,  for the smooth
  $\fz$-harmonic function
$$w(x)=2^{\frac13}x_1^{\frac43}- x_2^{\frac43}- x_3^{\frac43}\quad\mbox{in the domain $(0,\fz)^3$,} $$
a direct calculation gives
$$  |D^2w|^2-(\Delta w)^2=  \frac{32}{81} 2^{\frac23}x_1^{-\frac23}[x_2^{-\frac23} +x_3^{-\frac23}  ]-\frac{32}{81} x_2^{-\frac23}  x_3^{-\frac23},$$
which changes sign when $x_1$ goes from $0$ to $\fz$.
Considering this,  we conjecture that for  some/all $p\in(3+\frac2{n-2},\fz)$,
there exists
 a  $p$-harmonic function $u\in W^{2,2}_\loc$
such that
 $|D^2u|^2-(\Delta u)^2$ changes sign.
 \end{rem}

\section{Proof of  Theorem \ref{thm3}}

To prove  Theorem \ref{thm3}, it suffices to show
that  for any viscosity solution $u=u(x,t)$ to \eqref{pn-plap}, we have
\begin{equation}\label{W22-pn-plap} \bint_{Q_r }   [|D^2u |+|u_t|]^2  \,dx\,dt \le C(n,p ) \inf_{{{\vec c}}\in\rn}\frac1{r^2}\bint_{Q_{2r} }
  |Du-{{\vec c}}|^{2}  \,dx\,dt\quad\forall Q_{r}   \subset   Q_{2r} \Subset \Omega_T.
  \end{equation}
   Indeed, let $v(x)=u(x)- c_0-c_ix_i$ where
   $$c_0=\bint_{ Q_{2r}}  u \,dx\,dt \quad\mbox{and}\quad  \vec c=(c_1,\ldots,c_n)=  \bint_ {Q_{2r} } Du \,dx\,dt.$$
   We have $Dv=Du- \vec c$, $D^2v=D^2u$ and $v_t=u_t$. By the parabolic Sobolev-Poincar\'e inequality,
   \begin{align*}\|Du-\vec c\|_{L^2(Q_{2r})}&=\|Dv\|_{L^2(Q_{2r})}\\
   &\le C\frac1r\|u - c_0-c_ix_i\|_{L^{\frac{2(n+2)}{n+4}}(Q_{2r})}+C\|Du- \vec c\|_{L^{\frac{2(n+2)}{n+4}}(Q_{2r})}\\
   &\quad+
  Cr \|u_t\|_{L^{\frac{2(n+2)}{n+4}}(Q_{2r})}+ Cr \|D^2u \|_{L^{\frac{2(n+2)}{n+4}}(Q_{2r})}\\
  &\le Cr \|u_t\|_{L^{\frac{2(n+2)}{n+4}}(Q_{2r})}+ Cr \|D^2u \|_{L^{\frac{2(n+2)}{n+4}}(Q_{2r})},
   \end{align*}
   where in the last inequality we used \cite[Lemma 5.4]{k07}.
  This gives
\begin{align*}
&\left(\bint_{Q_r}    [|D^2u | +|u_t|]^2\,dx \,dt \right)^{1/2}\\
&\quad\le C(n,p )  \left( \bint_{{ Q_{2r}}}
  [|D^2u|+|u_t|]^{\frac{2(n+2)}{n+4}}  \,dx\,dt\right)^{\frac{n+4}{2(n+2)}}\quad\forall Q_{r}  \subset  Q_{2r} \Subset \Omega_T.
\end{align*}
 Since $\frac{2(n+2)}{n+4}<2$, by Gehring's lemma, there exists a $\dz_{n,p}>0$ such that
 $|D^2u|+|u_t|\in L^q_\loc(\Omega_T)$ for any $q<2+\dz_{n,p}$, and moreover, we have
 $$\left(\bint_{Q_r}   [|D^2u | +|u_t|]^q\,dx \,dt \right)^{1/q}\le C(n,p,q)  \left( \bint_{{ Q_{2r}}}
  [|D^2u|+|u_t|]^2  \,dx\,dt\right)^{\frac12} \quad\forall Q_{r}  \subset  Q_{2r} \Subset \Omega_T,$$
  which gives  \eqref{W2q-pn-plap} as desired.

To prove \eqref{W22-pn-plap}, given any fixed  smooth domain $U\Subset \Omega$, and for  $\ez\in(0,1]$,
let $u^\ez\in  C^0({ \overline {U_T}})$  be a viscosity  solution to   the regularized   equation
\eqref{ap-pn-plap}.
By the parabolic theory, we know that $u^\ez\in C^\fz(U_T )\cap C^0(\overline{U_T})$,
$Du^\ez\in L^\fz(U_T)$ uniformly in $\ez>0$ and $u^\ez\to u$ in $C^0(U_T)$  as $\ez\to0$; see \cite{js}.

{ In the sequel, without loss of generality, we assume that $Q_r=Q(0,r)\subset Q_{2r}\Subset U_T$. Take a smooth function $\phi\in C_c^\infty(B_{2r}\times (-4r^2,4r^2))$ such that
$$
\phi\equiv 1 \,\, \text{in}\,\,Q_r,\quad |\phi|\le 1,\quad |D\phi|\le \frac Cr,
\quad |D^2\phi|+|\phi_t|\le \frac C{r^2}.
$$}
Applying Lemma \ref{keylem}, we   prove the following.
The proof is postponed to the end of this section due to a technical reason.
\begin{lem}\label{du-cpll}
If   $n\ge 2$ and $1<p<3+\frac{2}{n-2}$, then we have
\begin{align}\label{xx1}
&\int_{{ Q_{2r}}}|D^2u^\ez|^2\phi^2\,dx\,dt+ \int_{{ Q_{2r}}}(u^\ez_t)^2\phi^2\,dx\,dt\nonumber\\
&\quad\le C(n,p)\inf_{{\vec c}\in\rn}\int_{{ Q_{2r}}}|Du^\ez-{{\vec c}}|^2[|D\phi|^2+|\phi||D^2\phi|+|\phi||\phi_t|]\,dx\,dt\nonumber\\
&\quad\quad+C(n,p)\ez\int_{{ Q_{2r}}}[1+|\ln[|D u^\ez|^2+\ez]|][|D\phi|^2+|\phi||\phi_t|]\,dx\,dt\nonumber\\
&\quad\quad+{ C(n,p)\ez\int_{B_{2r}}|\ln[|D u^\ez(x,0)|^2+\ez]|\phi^2(x,0)\,dx}.
\end{align}
\end{lem}

{ Given} Lemma \ref{du-cpll}, we prove  Theorem \ref{thm3} as below.
\begin{proof}[Proof of Theorem \ref{thm3}]
%{\color{red}By choosing suitable test function $\phi$, we know that}
  Lemma \ref{du-cpll}, together with $Du^\ez\in L^\fz(U_T)$ uniformly in $\ez>0$, implies that
$D^2 u^\ez,u^\ez_t\in L^2_\loc(U_T)$ uniformly in $\ez>0$. By  the compact parabolic embedding theorem,
$Du^\ez\to D u$ in $L^2_\loc(U_T)$,  and $D^2u^\ez\to D^2u$ and $u^\ez_t\to u_t$ weakly in $L^2_\loc(U_T)$ as $\ez\to0$.
Letting $\ez\to0$, we conclude \eqref{W22-pn-plap} from Lemma \ref{du-cpll} and arbitrariness of $U_T$.
So  we finish the proof of Theorem \ref{thm3}.
\end{proof}

Before we prove Lemma \ref{du-cpll}, we give a remark for parabolic Coders condition.

\begin{rem}\label{pn-plap-cor}\rm
Rewrite the equation \eqref{ap-pn-plap} as
$$u^\ez_t-\sum_{1\le i,j\le n}a^\ez_{ij}u^\ez_{x_ix_j}=0\quad\mbox{in $U_T$};\ u^\ez=u\ \mbox{on $\partial_ p  U_T$},$$
where the principle  coefficients
$$a^\ez_{ij}=\dz_{ij}+(p-2)\frac{u^\ez_{x_i}u^\ez_{x_j}}{|Du^\ez|^2+\ez}.$$
If $1<p<3+\frac2{n-1}$, then $\{a^\ez_{ij}\}_{1\le i,j\le n}$ satisfies  the parabolic Cordes condition (see e.g.  \cite[(1.106)]{mps}) uniformly in
  $\ez\in(0,1]$, that is,    there exists  $\dz>0$ such that
\begin{equation}\label{parcordes}\sum_{i,j=1}^n( a_{ij}^\epsilon )^2+1 \le \frac1{n+\dz}\left( \sum_{i=1}^n a_{ii}^\epsilon +1\right)^2\quad \mbox{  in $U_T$.}\end{equation}
But, if $ p\ge 3+\frac2{n-1}$,   \eqref{parcordes} does not necessarily holds.
Thus, only when  $1<p<3+\frac2{n-1}$, one may get the
   $W^{2,2}_\loc$-regularity in the spatial variables  and  the $W^{1,2}_\loc$-regularity in the time variable through the parabolic Cordes condition. However, even in this case a rigorous argument  cannot be found in the literature.
  \end{rem}

Finally, we prove Lemma \ref{du-cpll}. First,
 applying Lemma \ref{keylem} to $u^\ez$, we have
\begin{align*}  &   \frac n2{ |D^2 u^\ez D u^\ez|^2} -  {\Delta u^\ez \Delta_\fz u^\ez } -\frac{n-2}2
  (\Delta u^\ez)^2|Du^\ez|^2\le \frac{n-1}2[|D^2u^\ez|^2-(\Delta u^\ez)^2]|Du^\ez|^2.
 \end{align*}
 Dividing both sides by $|Du^\ez|^2+\ez$ and using \eqref{ap-pn-plap} we obtain
  \begin{align}\label{pi-pn-plap}
&\frac{n}{2}\frac{|D^2u^\ez Du^\ez|^2}{|Du^\ez|^2+\ez}+[\frac{1}{p-2}
-\frac{n-2}{2}](\Delta u^\ez)^2%+\frac{(n-2)\ez}{2}\frac{|D^2u^\ez|^2}{|Du^\ez|^2+\ez}
\nonumber\\
&\quad\le\frac{n-1}{2}[|D^2u^\ez|^2-(\Delta u^\ez)^2]+\frac{\ez}{2}\frac{(\Delta u^\ez)^2-|D^2u^\ez|^2}{|Du^\ez|^2+\ez}+ \frac{\Delta u^\ez u^\ez_t}{p-2}.
\end{align}
Moreover, since
\begin{align*}
\frac{|D^2u^\ez Du^\ez|^2}{|Du^\ez|^2+\ez}\ge[\frac{\Delta_\fz u^\ez}{|Du^\ez|^2+\ez}]^2=[-\frac{\Delta u^\ez}{p-2}+\frac{u^\ez_t}{p-2}
]^2=\frac{(\Delta u^\ez)^2}{(p-2)^2}+\frac{(u^\ez_t)^2}{(p-2)^2}-\frac{2\Delta u^\ez u^\ez_t}{(p-2)^2},
\end{align*}
\eqref{pi-pn-plap} leads to
\begin{align*}
 &[\frac n2 \frac1{(p-2)^2}+\frac{1}{p-2}
-\frac{n-2}{2}] (\Delta u^\ez)^2
+  \frac n2 \frac1{(p-2)^2}  ( u^\ez_t)^2\\
&\quad\le \frac{n-1}{2} [|D^2u^\ez|^2-(\Delta u^\ez)^2] +\frac{\ez}{2} \frac{(\Delta u^\ez)^2-|D^2u^\ez|^2}{|Du^\ez|^2+\ez}+[\frac{1}{p-2}+ \frac {n}{(p-2)^2}] \Delta u^\ez u^\ez_t.
\end{align*}
Adding both sides by
$$
[\frac n2 \frac1{(p-2)^2}+\frac{1}{p-2}
-\frac{n-2}{2}] [|D^2 u^\ez|^2-(\Delta u^\ez)^2]$$
we conclude that
\begin{align}\label{lpden2}
 &[\frac n2 \frac1{(p-2)^2}+\frac{1}{p-2}
-\frac{n-2}{2}] |D^2 u^\ez|^2
+  \frac n2 \frac1{(p-2)^2} ( u^\ez_t)^2
 \nonumber\\
&\quad\le[ \frac n2 \frac1{(p-2)^2}+\frac{1}{p-2}
+\frac{1}{2}]  [|D^2u^\ez|^2-(\Delta u^\ez)^2] \nonumber\\
&\quad\quad +\frac{\ez}{2} \frac{(\Delta u^\ez)^2-|D^2u^\ez|^2}{|Du^\ez|^2+\ez}+ [\frac{1}{p-2}+ \frac {n}{(p-2)^2}] \Delta u^\ez u^\ez_t.
\end{align}

By a parabolic version of Lemma \ref{div}, the integration of the first term on the right-hand sides of  \eqref{pi-pn-plap} and \eqref{lpden2}  with a test function can be handled as before.
But additional efforts are needed to treat
the integration of the second and third terms  on the right-hand sides in  \eqref{pi-pn-plap} and \eqref{lpden2} with a test function. Unlike \eqref{pi-plapx} and \eqref{alden1}, since we cannot divide by $|Du^\ez|^2$ here due to its possible vanishing, the additional second term on the right-hand sides in  \eqref{pi-pn-plap} and \eqref{lpden2}
always appear.

 Below we consider 2 cases:
\begin{enumerate}
\item[$\bullet$]
Case $1<p<\min\{6,3+\frac2{n-2}\}$. In this case  we use \eqref{lpden2} to prove \eqref{xx1}.  Note that when $n\ge 3$, we always have $3+\frac2{n-2}<6$.

\item[$\bullet$] Case $n=2$ and $p\ge 6$.  In this case  we use \eqref{pi-pn-plap} to prove \eqref{xx1}.
\end{enumerate}

\subsection{Case  $1<p<\min\{6,3+\frac2{n-2}\}$.}

Via a direct calculation we have the following.
\begin{lem} \label{vt-n} Let $p\in(1,2)\cup(2,\fz)$.
For any ${{\vec c}}\in\rn$,  we have
\begin{align}\label{vt-n-1}
 \int_{{ Q_{2r}}}\Delta u^\ez u^\ez _t\phi^2\,dx\,dt
&\le  \eta\int_{Q_{2r}}  |  D^2u^\ez |^2 \phi^2\,dx\,dt+ \frac  C \eta\int_{Q_{2r}}|Du^\ez-{{\vec c}}|^2[|D\phi|^2+ |\phi \phi_t|]\,dx\,dt.
\end{align}
\end{lem}

\begin{proof}
By integration by parts we have
\begin{align*}
\int_{Q_{2r}}\Delta u^\ez   u^\ez  _t\phi^2\,dx\,dt&=\int_{Q_{2r}}(u^\ez  _{x_i}-c_i)_{x_i} u^\ez  _t\phi^2\,dx\,dt\\
&=-\int_{Q_{2r}}(u^\ez  _{x_i}-c_i) u^\ez  _{x_it}\phi^2\,dx\,dt
-2\int_{Q_{2r}}(u^\ez  _{x_i}-c_i) u^\ez  _{t}\phi_{x_i}\phi\,dx\,dt.
\end{align*}
 Further integration by parts gives
$$-\int_{Q_{2r}}(u^\ez  _{x_i}-c_i) u^\ez  _{x_it}\phi^2\,dx\,dt=
-\frac{1}{2}\int_{Q_{2r}}(|Du^\ez-{{\vec c}}|^2)_t\phi^2\,dx\,dt
{ \le} \int_{Q_{2r}}|Du^\ez-{{\vec c}}|^2|\phi||\phi_t|\,dx\,dt$$
and
\begin{align*}-2\int_{Q_{2r}}(u^\ez  _{x_i}-c_i) u^\ez  _{t}\phi_{x_i}\phi\,dx\,dt
&\le 2\int_{Q_{2r}}|Du^\ez-{{\vec c}}| |u^\ez _t ||D\phi||\phi|\,dx\,dt\\
&\le \frac1{\eta}
\int_{Q_{2r}}|Du^\ez-\vec c|^2|D\phi|^2\,dx\,dt+\eta \int_{Q_{2r}} |u^\ez _t |^2\phi^2\,dx\,dt.
\end{align*}
 By the equation \eqref{ap-pn-plap},  we have
\begin{equation}
                                        \label{eq8.26}
|u^\ez_t|\le |\Delta u^\ez|+|p-2||D^2u^\ez|\le (p+n)|D^2u^\ez|.
\end{equation}
 Combining  all the estimations, we have \eqref{vt-n-1} as desired.
\end{proof}
Using this and the divergence structure of $[|D^2u^\ez|^2-(\Delta u^\ez)^2]$, we further have the following.
\begin{lem} \label{ez-n} Let  $ p\in(1,2)\cup(2,\fz)$. For any   $\eta\in(0,1)$, we have
\begin{align}\label{ez-n-1}
& \frac{\ez}{2}\int_{Q_{2r}}\frac{(\Delta u^\ez)^2-|D^2u^\ez|^2}{|Du^\ez|^2+\ez}\phi^2\,dx\,dt\nonumber\\
& \quad\le \frac{1}{ 4(p-1) }\int_{Q_{2r}}(u^\ez_t)^2\phi^2\,dx\,dt+ \eta  \int_{Q_{2r}} |D^2u^\ez|^2 \phi^2\,dx\,dt +
 C(\eta)\ez  \int_{Q_{2r}} |D\phi|^2\,dx\,dt.
\end{align}
\end{lem}
\begin{proof}[ Proof of Lemma \ref{ez-n}.]

By integration by parts, we obtain
\begin{align*}
&\frac{\ez}{2}\int_{Q_{2r}}\frac{ (\Delta u^\ez)^2-|D^2u^\ez|^2 }{|Du^\ez|^2+\ez}\phi^2\,dx\,dt\\
&\quad=-\frac{\ez}{2}\int_{Q_{2r}}[\Delta u^\ez u^\ez_{x_i}-u^\ez_{x_ix_j}u^\ez_{x_j}]
\left(\frac{\phi^2}{|Du^\ez|^2+\ez}\right)_{x_i}\,dx\,dt\\
&\quad=\ez\int_{Q_{2r}} \left(\Delta u^\ez \frac{\Delta_\fz u^\ez}{[|Du^\ez|^2+\ez]^2}-\frac{|D^2u^\ez Du^\ez|^2}{[|Du^\ez|^2+\ez]^2}\right) \phi^2\,dx\,dt
\\
&\quad\quad-\ez \int_{Q_{2r}}[\Delta u^\ez u^\ez_{x_i}-u^\ez_{x_ix_j}u^\ez_{x_j}]\phi_{x_i}\phi
\frac{1}{|Du^\ez|^2+\ez}\,dx\,dt.
\end{align*}
By Young's inequality we obtain
\begin{align*}
&\ez \int_{Q_{2r}}[\Delta u^\ez u^\ez_{x_i}-u^\ez_{x_ix_j}u^\ez_{x_j}]\phi_{x_i}\phi
\frac{1}{|Du^\ez|^2+\ez}\,dx\,dt\\
&\quad\le C\ez^{1/2}\int_{Q_{2r}}|D^2u^\ez||\phi D\phi|\,dx\,dt\\
 &\quad\le \eta  \int_{Q_{2r}} |D^2u^\ez|^2 \phi^2\,dx\,dt +
 C(\eta)\ez  \int_{Q_{2r}} |D\phi|^2\,dx\,dt.
\end{align*}
By  H\"older's inequality,  \eqref{ap-pn-plap},
 and Young's inequality one has
 \begin{align*}
&\ez\left(\Delta u^\ez \frac{\Delta_\fz u^\ez}{[|Du^\ez|^2+\ez]^2}-\frac{|D^2u^\ez Du^\ez|^2}{[|Du^\ez|^2+\ez]^2}\right)\\
&\quad\le \ez\left(\Delta u^\ez \frac{\Delta_\fz u^\ez}{[|Du^\ez|^2+\ez]^2}-\frac{(\Delta_\fz u^\ez)^2}{[|Du^\ez|^2+\ez]^3}\right)\\
&\quad=\frac{\ez}{(p-2)^2}\frac{1}{|Du^\ez|^2+\ez}
\left[(p-2)(\Delta u^\ez u_t-(\Delta u^\ez)^2)-(u_t^\ez-\Delta u^\ez)^2\right]\\
&\quad= \frac{\ez}{(p-2)^2}\frac{1}{|Du^\ez|^2+\ez}
\left[p\Delta u^\ez u_t -(p-1)(\Delta u^\ez)^2-(u_t^\ez)^2\right]\\
&\quad\le \frac{1}{4(p-1)}(u_t^\ez)^2.
\end{align*}
Combining all estimates together, we get \eqref{ez-n-1}.
\end{proof}

\begin{proof}[Proof of Lemma \ref{du-cpll}][ {\it Case   $1<p<\min\{6,3+\frac2{n-2}\}$.}]

By  Lemma \ref{div}, for any ${{\vec c}}\in\rn$  one gets
\begin{align}\label{e2.x4}
\left |\int_{Q_{2r}}[|D^2u^\ez|^2-(\Delta u^\ez)^2]\phi^2\,dx\,dt\right|\le C(n)\int_{Q_{2r}}|Du^\ez-{{\vec c}}|[|D\phi|^2+|D^2\phi||D\phi|]\,dx\,dt.
\end{align}

Multiplying both sides of \eqref{lpden2} by $\phi^2$ and integrating,
by \eqref{e2.x4}, Lemma \ref{ez-n} and Lemma \ref{vt-n}, for any $\vec c\in\rn$ and $\eta\in(0,1)$ we obtain
  \begin{align}
 &[\frac n2 \frac1{(p-2)^2}+\frac{1}{p-2}
-\frac{n-2}{2} -\eta]\int_{Q_{2r}}|D^2 u^\ez|^2\phi^2\,dx\,dt
 \nonumber\\
 &\quad\quad
+ [ \frac n2 \frac1{(p-2)^2} -\frac{1}{4 (p-1)}]\int_{Q_{2r}}( u^\ez_t)^2\phi^2\,dx\,dt
 \nonumber\\
&\quad\le   C(n,p,\eta)\int_{Q_{2r}}|Du^\ez-{{\vec c}}|^2[|\phi_t||\phi|+|D\phi|^2%+|\phi_t \phi|
]\,dx\,dt  +C(\eta)\ez  \int_{Q_{2r}} |D\phi|^2\,dx\,dt.\nonumber
\end{align}
Note that  $p\in (1,3+\frac2{n-2})$ implies that
 $$\frac n2 \frac1{(p-2)^2}+\frac{1}{p-2}
-\frac{n-2}{2}>0.$$
 Moreover, when $p\in (1+1/(n+1+\sqrt{n(n+2)}),n+2+\sqrt{n(n+2)})$, we have
$$\frac n2 \frac1{(p-2)^2} -\frac{1}{4(p-1)}>0.$$
Taking $\eta>0$  sufficiently small, and noting $n+2+\sqrt{n(n+2)}\ge 6$,
one has \eqref{xx1} under the condition that $p\in (1+1/(n+1+\sqrt{n(n+2)}),\min\{3+\frac2{n-2},6\})$.  Finally, it remains to notice by adding dummy variables $u$ also satisfies the equation in $\mathbb{R}^m$ for any $m\ge n$. Therefore, \eqref{xx1} holds for any $p$ in
$$
\bigcup_{m\ge n}(1+1/(m+1+\sqrt{m(m+2)}),\min\{3+\frac2{m-2},6\})
=(1, \min\{3+\frac2{n-2},6\}).
$$
 The lemma is proved in this case.
%%===================================================
%\medskip
%\noindent {\it Subcase $1<p<2$.} Note that
%$$\frac{\ez}{2} \frac{|D^2u^\ez|^2-(\Delta u^\ez)^2}{|Du^\ez|^2+\ez}\le \frac12(\Delta u^\ez)^2=
%-\frac12[|D^2u^\ez|^2-(\Delta u^\ez)^2 ]+\frac12|D^2u^\ez|^2 .$$
%From \eqref{lpden2} it follows that
%\begin{align*}
% &[\frac n2 \frac1{(p-2)^2}+\frac{1}{p-2}
%-\frac{n-2}{2}-\frac12] |D^2 u^\ez|^2
%+  \frac n2 \frac1{(p-2)^2} ( u^\ez_t)^2
% \nonumber\\
%&\quad\le[ \frac n2 \frac1{(p-2)^2}+\frac{1}{p-2}]  [|D^2u^\ez|^2-(\Delta u^\ez)^2] + [\frac{1}{p-2}+ \frac {n}{(p-2)^2}] \Delta u^\ez u^\ez_t
%\end{align*}
%Note that $1<p<2$ implies that
% $$\frac n2 \frac1{(p-2)^2}+\frac{1}{p-2}
%-\frac{n-2}{2}-\frac12>0.$$
%Using Lemma \ref{vt-n} and \eqref{e2.x4}, similarly to above subcase,   one has
%\eqref{xx1}.
\end{proof}

\subsection {Case $n=2$ and $6\le p<\fz$.}

 We note that the proofs in this subsection works for any $p>2$.
Instead of  Lemma  \ref{ez-n}  we have the following.
\begin{lem}\label{e7.x2} Let $n=2$ and $6\le p<\fz$, and let  $\phi\in C_c^\fz(U_T)$.
For any   $\eta\in(0,1)$ we have
\begin{align*}
&\frac{\ez}{2}\int_{Q_{2r}}\frac{(\Delta u^\ez)^2-|D^2u^\ez|^2}{|Du^\ez|^2+\ez}\phi^2\,dx\,dt\\
&\quad\le \frac{\ez}{p-2}\int_{Q_{2r}} \frac{u^\ez_t\Delta u^\ez}{|Du^\ez|^2+\ez}\phi^2\,dx\,dt
+\eta  \int_{Q_{2r}} |D^2u^\ez |^2 \phi^2\,dx\,dt+ \frac{C}{\eta } \ez\int_{Q_{2r}}|D\phi|^2\,dx\,dt.
\end{align*}
\end{lem}

\begin{proof}
The proof follows from that of Lemma \ref{ez-n} once we observe that
\begin{align*}
\ez\int_{Q_{2r}}\Delta u^\ez \frac{\Delta_\fz u^\ez}{[|Du^\ez|^2+\ez]^2}\phi^2\,dx\,dt
%&=\frac{\ez}{p-2}\int_{U_T} \frac{u^\ez_t\Delta u^\ez}{|Du^\ez|^2+\ez}\phi^2\,dx\,dt
%-\frac{\ez}{p-2}\int_{U_T} \frac{(\Delta u^\ez)^2}{|Du^\ez|^2+\ez}\phi^2\,dx\,dt\\
&\le \frac{\ez}{p-2}\int_{Q_{2r}} \frac{u^\ez_t\Delta u^\ez}{|Du^\ez|^2+\ez}\phi^2\,dx\,dt.
\end{align*}
\end{proof}

Moreover, instead of Lemma  \ref{vt-n}, we have the following, whose proof is postponed to the end of this subsection.
\begin{lem} \label{du-cpll-2} Let $n=2$ and $6\le p<\fz$.
For any ${{\vec c}}\in\rn$ and $\eta\in(0,1)$ we have
\begin{align*}%\label{les5}
&\frac{1}{p-2}\int_{Q_{2r}}\Delta u^\ez u^\ez_t\phi^2 \,dx\,dt\nonumber\\
&\quad\le -\frac{ \ez}{p-2}\int_{Q_{2r}} \frac{u^\ez_t\Delta u^\ez}{|Du^\ez|^2+\ez}\phi^2\,dx\,dt+ \int_{Q_{2r}}\frac{|D^2u^\ez Du^\ez|^2}{|Du^\ez|^2+\ez}\phi^2\,dx\,dt\\
&\quad\quad+
\eta \int_{Q_{2r}}|D^2u^\ez |^2\phi^2 \,dx\,dt+C(p,\eta)\int_{Q_{2r}}|D u^\ez-{{\vec c}}|^2[|\phi||D^2 \phi|+|D\phi|^2+|\phi||\phi_t|]\,dx\,dt\nonumber\\
&\quad\quad+\frac{C}{\eta} \ez\int_{Q_{2r}}|D\phi|^2\,dx\,dt
+C{\ez} \int_{Q_{2r}}|\ln [|D u^\ez|^2+\ez]||\phi||\phi_t|\,dx\,dt\\
&\quad\quad{ +C{\ez} \int_{B_{2r}}|\ln [|D u^\ez(x,0)|^2+\ez]|\phi^2(x,0)\,dx}.
\end{align*}
\end{lem}

\begin{proof}[Proof of Lemma \ref{du-cpll}][{\it  Case $n=2$ and $6\le p<\fz$.}]

Multiplying both sides of \eqref{pi-pn-plap} with $n=2$ by $\phi^2$ and integrating { in $Q_{2r}$},
by \eqref{e2.x4}, Lemma \ref{e7.x2} and Lemma \ref{du-cpll-2}, for any $\vec c\in\rn$ and $\eta\in(0,1)$ we obtain
\begin{align*}
&\frac{1}{p-2} \int_{Q_{2r}}(\Delta u^\ez)^2\phi^2\,dx\,dt-\eta \int_{Q_{2r}}|D^2u^\ez|^2\phi^2\,dx\,dt
%+3\eta\int_{Q_{2r}}(\Delta u^\ez)^2\phi^2[1-\frac{\ez}{|Du^\ez|^2+\ez}]\,dx\,dt
\\
&\quad\le C(p,\eta)\int_{Q_{2r}}|D u^\ez-\vec c|^2[|D\phi|^2+|\phi||D^2\phi| +|\phi||\phi_t|
]\,dx\,dt
+\frac{C}{\eta} \ez\int_{Q_{2r}}|D\phi|^2\,dx\,dt\\
&\quad\quad +C\ez\int_{Q_{2r}}|\ln [|D u^\ez|^2+\ez]||\phi||\phi_t|\,dx\,dt
{ +C{\ez} \int_{B_{2r}}|\ln [|D u^\ez(x,0)|^2+\ez]|\phi^2(x,0)\,dx}.
\end{align*}
Choosing $0<\eta<\frac1{2(p-2)}$ be sufficiently small,
  adding both sides by
$$\frac{1}{ p-2}\int_{Q_{2r}}[D^2u^\ez|^2-(\Delta u^\ez)^2]\phi^2\,dx\,dt,$$
and applying \eqref{e2.x4} and Lemma \ref{div}, we get
\begin{align*}
& \int_{Q_{2r}}|D^2u^\ez|^2\phi^2\,dx\,dt
\le C(p)\int_{Q_{2r}}|D u^\ez-\vec c|^2[%|\phi||\phi_t|+
|D\phi|^2+|\phi||D^2\phi|]\,dx\,dt\\
&\quad+C(p)\ez \int_{Q_{2r}}[|D\phi|^2+|\ln [|D u^\ez|^2+\ez]||\phi||\phi_t|]\,dx\,dt
{ +C(p)\ez \int_{B_{2r}}|\ln [|D u^\ez(x,0)|^2+\ez]|\phi^2\,dx}
\end{align*}  as desired.
\end{proof}

 Finally, we note that
Lemma \ref{du-cpll-2} follows from Lemmas
\ref{du-cpll-2-1} and \ref{du-cpll-2-2} below.

\begin{lem} \label{du-cpll-2-1} Let $n=2$ and $6\le p<\fz$.
For any $\eta>0$, we have
\begin{align}\label{les5}
&\frac{2 \ez}{p-2}\int_{Q_{2r}} \frac{u^\ez_t\Delta u^\ez}{|Du^\ez|^2+\ez}\phi^2\,dx\,dt\nonumber\\
&\quad\le \frac{1}{(p-2)^2}\int_{Q_{2r}}(u^\ez_t)^2\phi^2\,dx\,dt+
\ez\int_{Q_{2r}}\frac{|D^2u^{\ez}Du^\ez|^2}{[|Du^\ez|^2+\ez]^2}\phi^2\,dx\,dt+
\eta \int_{Q_{2r}}|D^2u^\ez |^2\phi^2 \,dx\,dt\nonumber\\
&\quad\quad+\frac{C}{\eta} \ez\int_{Q_{2r}}|D\phi|^2\,dx\,dt
+ \frac {2\ez}{p-2} \int_{Q_{2r}}|\ln [|D u^\ez|^2+\ez]||\phi||\phi_t|\,dx\,dt\nonumber\\
&\quad\quad { -\frac{\ez}{p-2} \int_{B_{2r}}\ln [|D u^\ez(x,0)|^2+\ez]\phi^2(x,0)\,dx}.
\end{align}
\end{lem}

\begin{lem} \label{du-cpll-2-2}Let $n=2$ and $6\le p<\fz$.
For any $\eta>0$, we have
\begin{align}\label{les6}
&\frac{1}{p-2}\int_{Q_{2r}}\Delta u^\ez u^\ez_t\phi^2\frac{|Du^\ez|^2}{|Du^\ez|^2+\ez}\,dx\,dt\nonumber\\
&\quad\le -\frac{1}{(p-2)^2}\int_{Q_{2r}}(u^\ez_t)^2\phi^2\,dx\,dt-
\ez\int_{Q_{2r}}\frac{|D^2u^{\ez}Du^\ez|^2}{[|Du^\ez|^2+\ez]^2}\phi^2\,dx\,dt+\int_{Q_{2r}}\frac{|D^2u^\ez Du^\ez|^2}{|Du^\ez|^2+\ez}\phi^2\,dx\,dt\nonumber\\
&\quad\quad+\eta \int_{Q_{2r}} |D^2u^\ez|^2\phi^2\,dx\,dt+C(p,\eta)\int_{Q_{2r}}|D u^\ez-{{\vec c}}|^2[|\phi||\phi_t|+|\phi||D^2\phi|+|D\phi|^2]\,dx\,dt\nonumber\\
&\quad\quad+\frac{C}{\eta} \ez\int_{Q_{2r}}|D\phi|^2\,dx\,dt
+\frac{\ez}{p-2}\int_{Q_{2r}}|\ln [|D u^\ez|^2+\ez]||\phi||\phi_t|\,dx\,dt\nonumber\\
&\quad\quad { +\frac{\ez}{2(p-2)} \int_{B_{2r}}\ln [|D u^\ez(x,0)|^2+\ez]\phi^2(x,0)\,dx}.
\end{align}
\end{lem}
\begin{proof}[Proof of Lemma \ref{du-cpll-2-1}]
By integration by parts we have
\begin{align*}
 &\frac{2 \ez}{p-2}\int_{Q_{2r}} \frac{u^\ez_t\Delta u^\ez}{|Du^\ez|^2+\ez}\phi^2\,dx\,dt\\
&\quad=-\frac{ 2\ez}{p-2}\int_{Q_{2r}} u^\ez_{x_i}\left(\frac{u^\ez_t}{|Du^\ez|^2+\ez}\phi^2\right)_{x_i}\,dx\,dt\\
&\quad=-\frac{ 2\ez}{p-2}\int_{Q_{2r}}\frac{u^\ez_{x_i}u^\ez_{x_it}}{|Du^\ez|^2+\ez}\phi^2\,dx\,dt
-\frac{4\ez}{p-2}\int_{Q_{2r}}\frac{u^\ez_{x_i}u^\ez_{t}\phi_{x_i}}{|Du^\ez|^2+\ez}\phi\,dx\,dt\\
&\quad\quad+\frac{4\ez}{p-2}\int_{Q_{2r}}\frac{\Delta_\fz u^\ez u^\ez_{t}}{[|Du^\ez|^2+\ez]^2}\phi^2\,dx\,dt.
\end{align*}

We estimate the three terms on the right-hand side in order.
First, from
 $$\frac{  u^\ez_{x_it}u^\ez_{x_i} }{|Du^\ez|^2+\ez}=\frac12\frac{(|Du^\ez|^2)_t }{|Du^\ez|^2+\ez}=\frac12 [\ln [|D u^\ez|^2+\ez]]_t $$
and integration by parts  it follows that
\begin{align}\label{lnterm}
&-\frac{ 2\ez}{p-2}\int_{Q_{2r}}\frac{u^\ez_{x_i}u^\ez_{x_it}}{|Du^\ez|^2+\ez}\phi^2\,dx\,dt
= -\frac{\ez}{p-2} \int_{Q_{2r}}[\ln [|D u^\ez|^2+\ez]]_t\phi^2\,dx\,dt\nonumber\\
&= \frac{2\ez}{p-2} \int_{Q_{2r}}\ln [|D u^\ez|^2+\ez]\phi\phi_t\,dx\,dt
{ -\frac{\ez}{p-2} \int_{B_{2r}}\ln [|D u^\ez(x,0)|^2+\ez]\phi^2(x,0)\,dx}.
\end{align}

Next, by Young's inequality  and \eqref{eq8.26}, one has
\begin{align}\label{constterm}
\left|\frac{4\ez}{p-2}\int_{Q_{2r}}\frac{u^\ez_{x_i}u^\ez_{t}\phi_{x_i}}{|Du^\ez|^2+\ez}\phi\,dx\,dt\right|& \le C\sqrt \ez \int_t|D^2u^\ez||\phi D\phi|\,dx\,dt\nonumber\\
&\le\eta  \int_{Q_{2r}}|D^2u^\ez|^2\phi^2 \,dx\,dt +\frac{C}{\eta} \ez\int_{Q_{2r}}|D\phi|^2\,dx\,dt.
\end{align}

Finally, by Young's inequality and noting
$$ 4\ez|Du^\ez|^2= (2\ez^{1/2}|Du^\ez|)^2 \le [|Du^\ez|^2+\ez]^2,$$
 we have
\begin{align*}
&\frac{4\ez}{p-2}\int_{Q_{2r}}\frac{\Delta_\fz u^\ez u^\ez_{t}}{[|Du^\ez|^2+\ez]^2}\phi^2\,dx\,dt\\
&\quad\le 4\ez^2\int_{Q_{2r}}\frac{|D^2u^\ez Du^\ez|^2|Du^\ez|^2}{[|Du^\ez|^2+\ez]^4}\phi^2\,dx\,dt
+\frac 1{(p-2)^2}\int_{Q_{2r}}(u^\ez_{t})^2\phi^2\,dx\,dt\\
&\quad\le    \ez\int_{Q_{2r}}\frac{|D^2u^\ez Du^\ez|^2}{[|Du^\ez|^2+\ez]^2}\phi^2\,dx\,dt+ \frac 1{(p-2)^2}\int_{Q_{2r}}(u^\ez_{t})^2\phi^2\,dx\,dt.
\end{align*}
Combining these estimates, we get  \eqref{les5}.
\end{proof}

\begin{proof}[Proof of  Lemma  \ref{du-cpll-2-2}]
By integration by parts, one gets
\begin{align*}
&\frac{1}{p-2}\int_{Q_{2r}}\Delta u^\ez u^\ez_t\phi^2\frac{|Du^\ez|^2}{|Du^\ez|^2+\ez}\,dx\,dt\\
&\quad=-\frac{1}{p-2}\int_{Q_{2r}}u^\ez_{x_i}\left(\frac{ u^\ez_t\phi^2|Du^\ez|^2}{|Du^\ez|^2+\ez}\right)_{x_i}\,dx\,dt\\
&\quad=-\frac{2}{p-2}\int_{Q_{2r}}\frac{\Delta_\fz u^\ez u^\ez_t\phi^2}{|Du^\ez|^2+\ez}\,dx\,dt
-\frac{1}{p-2}\int_{Q_{2r}}\frac{u^\ez_{x_i} u^\ez_{x_it}\phi^2|Du^\ez|^2}{|Du^\ez|^2+\ez}\,dx\,dt\\
&\quad\quad-\frac{2}{p-2}\int_{Q_{2r}}\frac{u^\ez_{x_i}\phi_{x_i} u^\ez_t\phi|Du^\ez|^2}{|Du^\ez|^2+\ez}\,dx\,dt
+\frac{2}{p-2}\int_{Q_{2r}}\frac{\Delta_\fz u^\ez u^\ez_t\phi^2|Du^\ez|^2}{[|Du^\ez|^2+\ez]^2}\,dx\,dt.
\end{align*}
Below, we bound the four terms on  the right-hand side in order.
First, from  Lemma \ref{vt-n}  it follows that
\begin{align*}
-\frac{2}{p-2}\int_{Q_{2r}}\frac{\Delta_\fz u^\ez u^\ez_t\phi^2}{|Du^\ez|^2+\ez}\,dx\,dt
&=-\frac{2}{(p-2)^2}\int_{Q_{2r}}(u^\ez_t)^2\phi^2\,dx\,dt
+\frac{2}{(p-2)^2}\int_{Q_{2r}}\Delta u^\ez u^\ez_t\phi^2\,dx\,dt\\
&\le -\frac{2}{(p-2)^2}\int_{Q_{2r}}(u^\ez_t)^2\phi^2\,dx\,dt+\eta \int_{Q_{2r}}|D^2 u^\ez|^2\phi^2\,dx\,dt\\
&\quad+
\frac{C}{(p-2)^2}\int_{Q_{2r}}|D u^\ez-{{\vec c}}|^2[|\phi||\phi_t|+|D\phi|^2]\,dx\,dt.
\end{align*}

 Next, by \eqref{lnterm}
\begin{align*}
 &-\frac{1}{p-2}\int_{Q_{2r}}\frac{u^\ez_{x_i}u^\ez_{x_it}\phi^2|Du^\ez|^2}{|Du^\ez|^2+\ez}\,dx\,dt\\
&\quad=-\frac{1}{p-2}\int_{Q_{2r}}u^\ez_{x_i}u^\ez_{x_it}\phi^2\,dx\,dt
+\frac{\ez}{p-2}\int_{Q_{2r}}\frac{\phi^2u^\ez_{x_i}u^\ez_{x_it}}{|Du^\ez|^2+\ez}\,dx\,dt\\
&\quad\le -\frac{1}{p-2}\int_{Q_{2r}}u^\ez_{x_i}u^\ez_{x_it}\phi^2\,dx\,dt+ \frac{\ez}{p-2} \int_{Q_{2r}}|\ln [|D u^\ez|^2+\ez]||\phi||\phi_t|\,dx\,dt\\
&\quad\quad { +\frac{\ez}{2(p-2)} \int_{B_{2r}}\ln [|D u^\ez(x,0)|^2+\ez]\phi^2(x,0)\,dx}.
\end{align*}

Moreover, by integration by parts,
\begin{align*}
&-\frac{2}{p-2}\int_{Q_{2r}}\frac{u^\ez_{x_i}\phi_{x_i} u^\ez_t\phi|Du^\ez|^2}{|Du^\ez|^2+\ez}\,dx\,dt\\
&\quad=-\frac{1}{p-2}\int_{Q_{2r}}u^\ez_{x_i}[\phi^2]_{x_i} u^\ez_t\,dx\,dt
+\frac{2\ez}{p-2}\int_{Q_{2r}}\frac{u^\ez_{x_i} \phi _{x_i} u^\ez_t}{|Du^\ez|^2+\ez}\phi\,dx\,dt\\
&\quad=\frac{1}{p-2}\int_{Q_{2r}}u^\ez_{x_i}u^\ez_{x_it}\phi^2 \,dx\,dt
+\frac{1}{p-2}\int_{Q_{2r}}(\Delta u^\ez)u^\ez_{t}\phi^2 \,dx\,dt+\frac{2\ez}{p-2}\int_{Q_{2r}}\frac{u^\ez_{x_i} \phi _{x_i} u^\ez_t}{|Du^\ez|^2+\ez}\phi\,dx\,dt.
\end{align*}
Applying Lemma \ref{div} and \eqref{constterm}, we get
\begin{align*}
&-\frac{2}{p-2}\int_{Q_{2r}}\frac{u^\ez_{x_i}\phi_{x_i} u^\ez_t\phi|Du^\ez|^2}{|Du^\ez|^2+\ez}\,dx\,dt\\
&\quad\le \frac{1}{p-2}\int_{Q_{2r}}u^\ez_{x_i}\phi^2 u^\ez_{x_it}\,dx\,dt
+C\int_{Q_{2r}}|Du^\ez-{{\vec c}}|^2[|D\phi|^2+|\phi||D^2\phi|+|\phi||\phi_t|] \,dx\,dt\\
&\quad\quad +\eta  \int_{Q_{2r}}|D^2u^\ez|^2\phi^2 \,dx\,dt +\frac{C}{\eta} \ez\int_{Q_{2r}}|D\phi|^2\,dx\,dt.
\end{align*}

 Finally, by  H\"older's inequality and  Young's inequality we obtain
\begin{align*}
 &\frac{2}{p-2}\int_{Q_{2r}}\frac{\Delta_\fz u^\ez u^\ez_t\phi^2|Du^\ez|^2}{[|Du^\ez|^2+\ez]^2}\,dx\,dt\\
&\quad\le \frac{2}{p-2}\int_{Q_{2r}}\frac{|D^2 u^\ez D u^\ez||Du^\ez||u^\ez_t|\phi^2 }{|Du^\ez|^2+\ez}\,dx\,dt\\
&\quad\le  \frac{1}{ (p-2)^2}\int_{Q_{2r}}(u^\ez_t)^2\phi^2\,dx\,dt
+ \int_{Q_{2r}}\frac{|D^2u^\ez Du^\ez|^2|Du^\ez|^2 }{[|Du^\ez|^2+\ez]^2 }\phi^2 \,dx\,dt \\
&\quad=\frac{1}{ (p-2)^2}\int_{Q_{2r}}(u^\ez_t)^2\phi^2\,dx\,dt+\int_{Q_{2r}}\frac{|D^2u^\ez Du^\ez|^2 }{ |Du^\ez|^2+\ez  }\phi^2 \,dx\,dt-\ez \int_{Q_{2r}}\frac{|D^2u^\ez Du^\ez|^2   }{[|Du^\ez|^2+\ez]^2 }\phi^2 \,dx\,dt.
\end{align*}
Combining above, we conclude  \eqref{les6}.
\end{proof}

\section{Proof of Theorem \ref{thm4}}

 Let $u=u(x,t) $ be a viscosity solution to  \eqref{p-plap}.
Given any   smooth domain $U\Subset \Omega$, and for any $\ez\in(0,1]$,
let $u^\ez\in W^{1,p}(U)\cap C^0(\overline U)$  be a weak solution to
\eqref{ap-par-plap}.
By the parabolic theory, it is known that $u^\ez\in C^\fz(U_T)$ and $u^\ez\to u$ in $C^0(U_T)$ as $\ez\to0$; see \cite{df85,w86} for example.

Using the divergence structure of \eqref{ap-par-plap}, one easily gets the following.
\begin{lem} \label{th5-1} Let $p\in(1,2)\cup(2,\fz)$.  { Then} we have
\begin{align}\label{uezt-1}
\int_{Q_{2r}}(u^\ez_t)^2\phi^2\,dx\,dt\le C
\int_{Q_{2r}}[|Du^\ez|^2+\ez]^{\frac{p}2}|\phi||\phi_t|\,dx\,dt+
C\int_{Q_{2r}}[|Du^\ez|^2+\ez]^{{p-1}}|D\phi|^2\,dx\,dt.
\end{align}\end{lem}

\begin{proof}
By integration by parts we obtain \begin{align*}
&\int_{Q_{2r}}(u^\ez_t)^2\phi^2\,dx\,dt=\int_{Q_{2r}}u^\ez_t{\rm div}([|Du^\ez|^2+\ez]^{\frac{p-2}2}Du)\phi^2\,dx\,dt\\
&\quad=-\int_{Q_{2r}}u^\ez_{x_it}u^\ez_{x_i}[|Du^\ez|^2+\ez]^{\frac{p-2}2}\phi^2\,dx\,dt
-2\int_{Q_{2r}}u^\ez_{t}u^\ez_{x_i}\phi_{x_i}\phi[|Du^\ez|^2+\ez]^{\frac{p-2}2}\,dx\,dt
\end{align*}
 and
\begin{align*}
-\int_{Q_{2r}}u^\ez_{x_it}u^\ez_{x_i}[|Du^\ez|^2+\ez]^{\frac{p-2}2}\phi^2\,dx\,dt
&=-\frac{1}{p}\int_{Q_{2r}}([|Du^\ez|^2+\ez]^{\frac{p}2})_t\phi^2\,dx\,dt\\
&{ \le} \frac{2}{p}\int_{Q_{2r}}[|Du^\ez|^2+\ez]^{\frac{p}2}\phi\phi_t\,dx\,dt.
\end{align*}
By Young's inequality, we have
\begin{align*}
&-2\int_{Q_{2r}}u^\ez_{t}u^\ez_{x_i}\phi_{x_i}\phi[|Du^\ez|^2+\ez]^{\frac{p-2}2}\,dx\,dt \le \frac{1}{2}\int_{Q_{2r}}(u^\ez_t)^2\phi^2\,dx\,dt
+C\int_{Q_{2r}}|D\phi|^2[|Du^\ez|^2+\ez]^{{p-1}}\,dx\,dt.
\end{align*}
Combining the above estimates, we obtain \eqref{uezt-1}.
 \end{proof}

Next, we have the following, whose proof is postponed to the end of this section.
\begin{lem}\label{th5p3}
  { For any}  $p\in(1,2)\cup(2,3)$,  we have
\begin{align}\label{duez-1}
\int_{Q_{2r}}|D^2u^\ez|^2\phi^2\,dx\,dt
& \le C(n,p)\int_{Q_{2r}}|Du^\ez |^2[|D\phi|^2+|\phi ||D^2\phi|]\,dx\,dt\nonumber\\
& \quad+  C(n,p) \int_{Q_{2r}}[|Du^\ez |^2+\ez]^{\frac{4-p}2} |\phi||\phi_t| \,dx\,dt.
\end{align}

\end{lem}
 Theorem \ref{thm4} then follows from Lemmas \ref{th5-1} and \ref{th5p3} as follows.
\begin{proof}[Proof of Theorem \ref{thm4}]  { From Lemmas \ref{th5-1}, \ref{th5p3},  and $Du^\ez\in L^\fz(U_T)$ uniformly in $\ez>0$,} we conclude { that} $D^2 u^\ez,u^\ez_t\in L^2_\loc(U_T)$ uniformly in $\ez>0$.  By   the parabolic compact embedding theorem,
$Du^\ez\to D u$ in $L^2_\loc(U_T)$,  and $D^2u^\ez\to D^2u$ and $u^\ez_t\to u_t$ weakly in $L^2_\loc(U_T)$ as $\ez\to0$.
Letting $\ez\to0$, we conclude the proof of Theorem \ref{thm4}.
\end{proof}

 \begin{rem} \label{p-plap-cor}\rm
Due to the possible degeneracy of $Du^\ez$
when $p\ne 2$, one cannot expect the parabolic Coders condition \eqref{parcordes} holds for the principle coefficients of the equation \eqref{ap-par-plap} uniformly in $\ez>0$.
\end{rem}

Finally, we
 prove Lemma \ref{th5p3}. Firstly we  derive the following inequality from \eqref{keyineqx}:
\begin{align}\label{pi-pn-plap-1}
&[\frac{n}{2(p-2)^2}-\frac n2]|D^2u^\ez|^2   \nonumber  \\
&\quad\le[ \frac{n}{2(p-2)^2}-\frac 12][|D^2u^\ez|^2-(\Delta u^\ez)^2]\nonumber\\
&\quad\quad- \frac{n-2p+4}{(p-2)^2} [\frac12(u^\ez_t)^2(|Du^\ez|^2+\ez)^{2-p}- \Delta u^\ez u^\ez_t(|Du^\ez|^2+\ez)^{\frac{2-p}2}]\nonumber \\
&\quad\quad+\{-(\Delta u^\ez)^2-(p-2)\frac{(\Delta_\fz u^\ez)^2}{[|Du^\ez|^2+\ez]^2}+\frac{\Delta u^\ez u^\ez_t}{p-2}[|Du^\ez|^2+\ez]^{\frac{2-p}2}+ \frac{\ez}{2}\frac{(\Delta u^\ez)^2}{|Du^\ez|^2+\ez}\}.
\end{align}
Indeed, applying  \eqref{keyineqx} to $u^\ez$ and using \eqref{ap-par-plap}, similarly to \eqref{pi-pn-plap}   we have
 \begin{align*}
 & \frac{n}{2}\frac{|D^2u^\ez D u^\ez|^2}{|Du^\ez|^2+\ez}+[\frac{1}{p-2}
-\frac{n-2}{2}](\Delta u^\ez)^2%+\frac{(n-2)\ez}{2}\frac{|D^2u^\ez|^2}{|Du^\ez|^2+\ez}
\nonumber\\
&\quad\le\frac{n-1}{2}[|D^2u^\ez|^2-(\Delta u^\ez)^2]+\frac{\Delta u^\ez u^\ez_t}{p-2}[|Du^\ez|^2+\ez]^{\frac{2-p}2}+ \frac{\ez}{2}\frac{(\Delta u^\ez)^2}{|Du^\ez|^2+\ez}.
\end{align*}
Thus,  by using H\"older's inequality and rearranging terms,
 \begin{align*}
 & [\frac{n}{2}-(p-2)]\frac{(\Delta_\fz u^\ez)^2}{[|Du^\ez|^2+\ez]^2}+[\frac{1}{p-2}
- \frac{n}{2} ](\Delta u^\ez)^2%+\frac{(n-2)\ez}{2}\frac{|D^2u^\ez|^2}{|Du^\ez|^2+\ez}
\nonumber\\
&\quad\le\frac{n-1}{2}[|D^2u^\ez|^2-(\Delta u^\ez)^2]\\
&\quad\quad-(\Delta u^\ez)^2-  (p-2)\frac{(\Delta_\fz u^\ez)^2}{[|Du^\ez|^2+\ez]^2}+\frac{\Delta u^\ez u^\ez_t}{p-2}[|Du^\ez|^2+\ez]^{\frac{2-p}2}+ \frac{\ez}{2}\frac{(\Delta u^\ez)^2}{|Du^\ez|^2+\ez}.
\end{align*}
 Using \eqref{ap-par-plap} again, we have
  \begin{align*} & [\frac{n}{2}-(p-2)]\frac{(\Delta_\fz u^\ez)^2}{[|Du^\ez|^2+\ez]^2}+[\frac{1}{p-2}
- \frac{n}{2} ](\Delta u^\ez)^2%+\frac{(n-2)\ez}{2}\frac{|D^2u^\ez|^2}{|Du^\ez|^2+\ez}
\nonumber\\
& =[\frac{n}{2(p-2)^2}-\frac n2]  (\Delta u^\ez)^2 + \frac{n-2p+4}{(p-2)^2} [\frac12(u^\ez_t)^2(|Du^\ez|^2+\ez)^{2-p}-\Delta u^\ez u^\ez_t(|Du^\ez|^2+\ez)^{\frac{2-p}2}].
\end{align*}
Write
\begin{align*}[\frac{n}{2(p-2)^2}-\frac n2]  (\Delta u^\ez)^2
 & =[\frac{n}{2(p-2)^2}-\frac n2]  |D^2u^\ez|^2 -[\frac{n}{2(p-2)^2}-\frac n2] [|D^2u^\ez|^2- (\Delta u^\ez)^2].
\end{align*}
We therefore obtain \eqref{pi-pn-plap-1}.

Next, we have the following  estimate for  the  second term on the right-hand side of \eqref{pi-pn-plap-1}.
 \begin{lem} \label{th5-1-1} Let $p\in(1,2)\cup(2,{ 4})$. For  any $\eta>0$,
 we have
 \begin{align}\label{uezt-2}
&  \left|\int_{Q_{2r}}u^\ez_t \Delta u^\ez [|D u^\ez|^2+\ez]^{\frac{2-p}2}\phi^2\,dx\,dt-\frac12\int_{Q_{2r}}(u^\ez_{t})^2       [|D u^\ez|^2+\ez]^{  {2- p }}\phi^2\,dx\,dt\right|\nonumber\\
&\quad\le \eta
 \int_{Q_{2r}}|D^2u^\ez|^2\phi^2\,dx\,dt+C(\eta)\int_{Q_{2r}}  |Du^\ez|^2 |D\phi|^2 \,dx\,dt\nonumber\\ &\qquad  +\frac {1}{ 4-p }\int_{Q_{2r}}  [|D u^\ez|^2+\ez]^{\frac{4-p}2} |\phi||\phi_t|\,dx\,dt.
\end{align}
 \end{lem}
 \begin{proof}
By integration by parts, one has
\begin{align*}
&  \int_{Q_{2r}}u^\ez_t \Delta u^\ez [|D u^\ez|^2+\ez]^{\frac{2-p}2}\phi^2\,dx\,dt\\
&\quad=-\int_{Q_{2r}}u^\ez_{tx_i} u^\ez_{x_i}[|D u^\ez|^2+\ez]^{\frac{2-p}2}\phi^2\,dx\,dt
+(p-2)
 \int_{Q_{2r}}u^\ez_{t}    \frac{\Delta_\fz  u^\ez } {|Du^\ez|^2+\ez }[|D u^\ez|^2+\ez]^{ \frac{2- p }2}\phi^2\,dx\,dt \\
&\quad\quad - 2\int_{Q_{2r}}u^\ez_{t}\phi\phi_{x_i} u^\ez_{x_i}[|D u^\ez|^2+\ez]^{\frac{2-p}2}\,dx\,dt.
\end{align*}
By \eqref{ap-par-plap} we have
\begin{align*}
&(p-2)
 \int_{Q_{2r}}u^\ez_{t}    \frac{\Delta_\fz  u^\ez } {|Du^\ez|^2+\ez }[|D u^\ez|^2+\ez]^{ \frac{2- p }2}\phi^2\,dx\,dt\\
 &\quad=
 \int_{Q_{2r}}(u^\ez_{t})^2       [|D u^\ez|^2+\ez]^{  {2- p }}\phi^2\,dx\,dt-
 \int_{Q_{2r}}u^\ez_{t} \Delta u^\ez      [|D u^\ez|^2+\ez]^{ \frac{2- p }2}\phi^2\,dx\,dt.
 &\end{align*}
 Adding the above two equalities gives
 \begin{align*}
&  \int_{Q_{2r}}u^\ez_t \Delta u^\ez [|D u^\ez|^2+\ez]^{\frac{2-p}2}\phi^2\,dx\,dt-\frac12
 \int_{Q_{2r}}(u^\ez_{t})^2       [|D u^\ez|^2+\ez]^{  {2- p }}\phi^2\,dx\,dt\\
&\quad=-\frac1{2}\int_{Q_{2r}}u^\ez_{tx_i} u^\ez_{x_i}[|D u^\ez|^2+\ez]^{\frac{2-p}2}\phi^2\,dx\,dt
  -  \int_{Q_{2r}}u^\ez_{t}\phi\phi_{x_i} u^\ez_{x_i}[|D u^\ez|^2+\ez]^{\frac{2-p}2}\,dx\,dt.
\end{align*}
Noting that  by \eqref{ap-par-plap},
\begin{equation}
                                    \label{eq8.37}
|u^\ez_{t}|[|D u^\ez|^2+\ez]^{\frac{2-p}2}\le (p+n)|D^2u^\ez|.
\end{equation}
By Young's inequality  and \eqref{eq8.37}, one has
 \begin{align*}&\left| \int_{Q_{2r}}u^\ez_{t}\phi\phi_{x_i} u^\ez_{x_i}[|D u^\ez|^2+\ez]^{\frac{2-p}2}\,dx\,dt\right|\\
 &\quad \le \eta \int_{Q_{2r}}(u^\ez_{t})^2[|D u^\ez|^2+\ez]^{2-p}\phi^2 \,dx\,dt+
C(\eta)\int_{Q_{2r}}  |Du^\ez|^2 |D\phi|^2 \,dx\,dt\\
 &\quad \le (p+n) \eta\int_{Q_{2r}}|D^2u^\ez|^2\phi^2 \,dx\,dt+
C(\eta)\int_{Q_{2r}}  |Du^\ez|^2 |D\phi|^2 \,dx\,dt\end{align*}
and, by integration by parts,  \begin{align*}
-\frac1{2}\int_{Q_{2r}}u^\ez_{tx_i} u^\ez_{x_i}[|D u^\ez|^2+\ez]^{\frac{2-p}2}\phi^2\,dx\,dt&= -\frac1{2  (4-p) }\int_{Q_{2r}} ([|D u^\ez|^2+\ez]^{\frac{4-p}2})_t\phi^2\,dx\,dt\\
&\le  \frac1{ 4-p }\int_{Q_{2r}}  [|D u^\ez|^2+\ez]^{\frac{4-p}2} |\phi||\phi_t|\,dx\,dt.
\end{align*}
Combining all estimates above, we obtain \eqref{uezt-2}.
\end{proof}

Moreover, the last term  on the right-hand  of \eqref{pi-pn-plap-1} will be estimated as follows.

\begin{lem}  \label{th}Let $p\in(1,2)\cup(2,\fz)$. For any $\eta>0$,  we have
\begin{align}\label{pcla1}
&-\int_{Q_{2r}}(\Delta u^\ez)^2
 \phi^2\,dx\,dt-(p-2)\int_{Q_{2r}}  \frac{(\Delta_\fz u^\ez)^2}{[|Du^\ez|^2+\ez]^2}  \phi^2\,dx\,dt \nonumber\\
 &\quad\quad+\frac{1}{p-2}\int_{Q_{2r}}u_t\Delta u^\ez[|Du^\ez|^2+\ez]^{\frac{2-p}2}\phi^2\,dx\,dt
 +\frac{\ez}{2}\int_{Q_{2r}}\frac{(\Delta u^\ez)^2}{|Du^\ez|^2+\ez}\phi^2\,dx\,dt\nonumber\\
 &\quad\le    \eta \int_{Q_{2r}}|D^2u^\ez|^2 \phi^2 \,dx\,dt + C(n,p,\eta)\int_{Q_{2r}} |Du^\ez|^2[|D^2\phi|+|D\phi|^2]\,dx\,dt\nonumber\\
 &\quad\quad
+C(n,p)\int_{Q_{2r}}[|Du^\ez|^2+\ez]^{\frac{4-p}2}|\phi||\phi_t|\,dx\,dt.
\end{align}
 \end{lem}

 \begin{proof}   First, by using integration by parts and Young's inequality,
\begin{align*}
&  \frac1{p-2}\int_{Q_{2r}}u^\ez_t \Delta u^\ez [|D u^\ez|^2+\ez]^{\frac{2-p}2}\phi^2\,dx\,dt\\
&\quad\le
 \int_{Q_{2r}}u^\ez_{t}    \frac{\Delta_\fz  u^\ez } {|Du^\ez|^2+\ez }[|D u^\ez|^2+\ez]^{ \frac{2- p }2}\phi^2\,dx\,dt \\
 &\quad\quad+\eta \int_{Q_{2r}}(u^\ez_{t})^2[|D u^\ez|^2+\ez]^{2-p}\phi^2 \,dx\,dt+
C(\eta)\int_{Q_{2r}}  |Du^\ez|^2 |D\phi|^2 \,dx\,dt\\
&\quad\quad+ C\int_{Q_{2r}}  [|D u^\ez|^2+\ez]^{\frac{4-p}2}  |\phi||\phi_t| \,dx\,dt.
\end{align*}
By \eqref{ap-par-plap}, we write
\begin{align*}
  u^\ez_{t}    \frac{\Delta_\fz  u^\ez } {|Du^\ez|^2+\ez }[|D u^\ez|^2+\ez]^{ \frac{2- p }2}
  &=(p-2)   \frac{(\Delta_\fz u^\ez)^2}{[|Du^\ez|^2+\ez]^2}   +   \frac{\Delta_\fz u^\ez \Delta u^\ez}{|Du^\ez|^2+\ez}.
  \end{align*}
 By H\"older inequality and Young's inequality,
  \begin{align*}
 \frac{\Delta_\fz u^\ez \Delta u^\ez}{|Du^\ez|^2+\ez}&
 \le \frac{|D^2 u^\ez||Du^\ez|^2|\Delta u^\ez|}{|Du^\ez|^2+\ez} \le
 \frac12 \frac{ ( \Delta u^\ez)^2|Du^\ez|^2}{|Du^\ez|^2+\ez}
   +   \frac12  |D^2u^\ez|^2  \\
   & =    ( \Delta u^\ez)^2
   +\frac12[|D^2u^\ez|^2 -(\Delta u^\ez)^2 ]- \frac12 \ez\frac{ ( \Delta u^\ez)^2  }{|Du^\ez|^2+\ez}.
\end{align*}
Noting by \eqref{eq8.37},
$$|u^\ez_{t}|[|D u^\ez|^2+\ez]^{\frac{2-p}2}\le (p+n)|D^2u^\ez|, $$
and using Lemma \ref{div}, we conclude
\eqref{pcla1}.
\end{proof}

Using Lemma \ref{div}, Lemmas \ref{th} and \ref{th5-1-1}, we  prove Lemma \ref{th5p3} as follows.

\begin{proof} [Proof of Lemma \ref{th5p3}.]
Note that $p\in(1,2)\cup(2,3)$ implies
$$\frac n{2(p-2)^2}-\frac n2>0.$$
Choose $\eta(n,p) = \frac14[\frac n{2(p-2)^2}-\frac n2]$.
Multiplying both sides of \eqref{pi-pn-plap-1} by $\phi^2$,
by Lemmas \ref{th} and \ref{th5-1-1} with $\eta=\eta(n,p)$ and by Lemma \ref{div},   we have   \eqref{duez-1} as desired.
 \end{proof}

%\medskip
% \noindent {\bf Acknowledgment}. The authors%
%% would like to thank  Professor Yifeng Yu for several valuable discussions
%% of this paper. Wang is partially supported by NSF DMS 1764417.
% %Zhou is
% are partially supported by National Natural Science Foundation of China (No. 11522102, 11871088).

\noindent  Hongjie Dong

\noindent  Division of Applied Mathematics, Brown University,
182 George Street, Providence, RI 02912, USA

\noindent{\it E-mail }:  \texttt{hongjie\_dong@brown.edu}
\bigskip

\noindent Peng Fa

\noindent
Department of Mathematics, Beihang University, Haidian District Xueyuan Road  No.37,  Beijing 100191, P. R. China

\noindent{\it E-mail }:  \texttt{pengfa@buaa.edu.cn}
\bigskip

\noindent Yi Zhang

\noindent ETH Zurich, Department of Mathematics, Ramistrasse 101, 8092 Zurich, Switzerland
%and Hausdorff Center for Mathematics, Endenicher Allee 62, Bonn 53115, Germany

\noindent{\it E-mail}: \texttt{yizhang3@ethz.ch}

\bigskip

\noindent  Yuan Zhou

\noindent
Department of Mathematics, Beihang University, Haidian District Xueyuan Road  No.37, Beijing 100191, P. R. China

\noindent{\it E-mail }:  \texttt{yuanzhou@buaa.edu.cn}
%

%
%
%
%\noindent Peng Fa
%
%
%\noindent
%Department of Mathematics, Beihang University, Haidian District Xueyuan Road  No.37,  Beijing 100191, P. R. China
%
%\noindent{\it E-mail }:  \texttt{pengfa@buaa.edu.cn}
%\bigskip
%
%
%\noindent  Yuan Zhou
%
%
%\noindent
%Department of Mathematics, Beihang University, Haidian District Xueyuan Road  No.37, Beijing 100191, P. R. China
%
%\noindent{\it E-mail }:  \texttt{yuanzhou@buaa.edu.cn}
%

\end{document}